\documentclass[reqno]{amsart}
\usepackage{amsfonts}
\usepackage{amsmath}
\usepackage{graphics}
\usepackage{ dsfont}
\newcommand{\nr}{{|\!|}}
\newcommand{\st}{\int_0^t}
\newcommand{\s}{\int_{\Omega}}

\newcommand{\0}{{2}}
\newcommand{\1}{{H^1(\Omega)}}

\newcommand{\car}{{\mathds 1_{\{b<0\}}}}
\newtheorem{thm}{Theorem}[section]
\newtheorem{prop}[thm]{Proposition}
\newtheorem{lem}[thm]{Lemma}

\newtheorem{remark}[thm]{Remark}

\newcommand{\rem}[1]{}

\begin{document}

\title[3D Brinkman-Forchheimer-extended Darcy Model]
{Continuous data assimilation for the three-dimensional Brinkman-Forchheimer-extended Darcy Model
}
\date{February 3, 2015}

\author[P. A. Markowich]{Peter A. Markowich}
\address[P. A. Markowich]
{Division of Math and Computer Sci. and Eng. \\
 King Abdullah University of Science and Technology \\
Thuwal 23955-6900 \\
Saudi Arabia}
\email{peter.markowich@kaust.edu.sa}

\author[E.S. Titi]{Edriss S. Titi}
\address[E.S. Titi]
{Department of Mathematics \\
Texas A\&M University \\
3368-TAMU\\
College Station, TX 77843-3368 \\
USA
}
\email{titi@math.tamu.edu}

\author[S. Trabelsi]{Saber Trabelsi}
\address[S. Trabelsi]
{ Division of Math and Computer Sci. and Eng. \\
 King Abdullah University of Science and Technology \\
Thuwal 23955-6900 \\
Saudi Arabia
}
\email{saber.trabelsi@kaust.edu.sa}

\begin{abstract}
In this paper we introduce and analyze an algorithm for continuous data assimilation for a three-dimensional Brinkman-Forchheimer-extended Darcy (3D BFeD) model of porous media. This model is believed to be accurate when the flow velocity is too large for Darcy's law to be valid, and additionally the porosity is not too small. The algorithm is inspired by ideas developed for designing finite-parameters feedback control for dissipative systems. It aims to obtaining improved estimates of the state of the physical system by incorporating deterministic or noisy measurements and observations. Specifically, the algorithm involves a feedback control that nudges the large scales of the approximate solution toward those of the reference solution associated with the spatial measurements. In the first part of the paper, we present few results of existence and uniqueness of weak and strong solutions of the 3D BFeD system. The second part is devoted to the setting and convergence analysis of the data assimilation algorithm.
\end{abstract}
\maketitle

\section{Introduction}  \label{S-1}
The mathematical modeling and analysis of nonlinear flows and transport processes through a porous media is a very active field of research in mathematics and physics due to the challenging problems in engineering and applied sciences it covers. Most models of porous media are based on Darcy's law so ``{\it Darcy's equation has become the model of choice for the study of the flow of fluids through porous solids due to the pressure gradients, so much that it has now been elevated to the status of a law in physics}" (see, e.g., \cite{Raja}). Darcy's empirical flow model represents a simple linear relationship between flow rate and the pressure drop in a porous media
\[u_f=- \frac k\mu\, \nabla p, \]
where $u_f$ is the Darcy velocity, $k$ is the permeability of the porous medium, $\mu$ is the dynamic viscosity of the fluid, and $p$  the pressure. Any deviation from this scenario is termed non-Darcy flow. Roughly speaking, this law neglects the inertia or the acceleration forces in the fluid when compared to the classical Navier-Stokes equations. Also, it assumes that in a porous medium, a large body with large surface area of the pores, is exposed to the fluid flow so that the viscous resistance will greatly exceed acceleration forces in the fluid unless turbulence sets in. However, there are several situations where nature deviates from Darcy's law, for instance when one deals with high velocity, molecular and ionic effects or in the presence of some non-Newtonian fluids phenomena. In these situations, it is imperative to develop a more adequate models than the classical models based on Darcy's law. In practice, this can be achieved using Forchheimer equation based on an observation of P. Forchheimer which states that the relationship between the flow rate and pressure gradient is nonlinear at sufficiently high velocity and that this nonlinearity increases with flow rate. The Darcy-Forchheimer law states
\[ \nabla p= -\frac{\mu}{k}\,{\bf v}_f- \gamma \rho_f |{\bf v}_f|^2 \,{\bf v}_f,\]
where $\gamma>0$ is the so-called Forchheimer coefficient and ${\bf v}_f$ stands for the Forchheimer velocity and $\rho_f$ the density. In other words, Forchheimer law assumes that Darcy's law is still valid up to an additional nonlinear term  to account for the increased pressure drop.
\vskip6pt
\noindent The BFeD model is then based on a Darcy-Forchheimer  law. It  was originally derived in its classical configuration ($\alpha=\frac12, \beta=0, a>0$, and $b>0$, see equation \eqref{oursystem} below) in the framework of thermal dispersion in a porous medium using the method of volume averaging of the velocity and temperature deviations in the pores (see e.g. \cite{Hsu}). A discussion of the formulation, validity and limitation of the BFD system can be found in \cite{Vafai,Nield}. In this paper, we consider the following mathematical generalization of the BFeD model
\begin{equation}{\label{oursystem}}
\left\lbrace\begin{array}{ll}
&\partial_t\,u-\nu\,\Delta\,u +(u\cdot\nabla)\,u +\nabla\,p+a\,|u|^{2\alpha}u+b\,|u|^{2\beta}u=f,\\ \\
& \nabla\cdot u=0,,\;u\vert_{t=0}=u_0,
\end{array}
\right.
\end{equation}
subjected either to periodic boundary conditions, with period $L$, and $\Omega=[0,L]^3$ is the basic periodic domain
\begin{equation}\label{pbc} \left\lbrace\begin{array}{lcl}
&u(x+L,y,z,t)=u(x,y+L,z,t)=u(x,y,z+L,t) =u(x,y,z,t),\\ &\\
&p(x+L,y,z,t)=p(x,y+L,z,t)=p(x,y,z+L,t) =p(x,y,z,t),\end{array}\right.
\end{equation}
or Dirichlet no-slip boundary conditions
\begin{equation} \label{dirichlet} u\vert_{\partial \Omega}=0, \end{equation}
where $\partial \Omega$ denotes the boundary of a smooth domain $\Omega$, $\alpha>\beta\geq0$ are constants, and $a$ and $b$ are real numbers. The unknowns are the velocity field, $u$, and the pressure, $p$ and $f$ is a given forcing term. In system \eqref{oursystem}, we introduced the extra term $b\,|u|^{2\beta}u$ to model a pumping, when $b<0$, by opposition to the damping modeled through the term $a\,|u|^{2\alpha}u$ when $a>0$. Notice that in the limit case $a=b=0$, we obtain the classical Navier-Stokes system.  To our knowledge, there are only few mathematical results concerning this model and most of them focus on the case $\frac12\leq \alpha \leq 1$ and $\beta=0$. The continuous dependence on Brinkman and Forchheimer coefficients and the convergence as $\nu \to 0$ of solutions of DBF equation to the solutions of
\[\partial_t\,u+(u\cdot\nabla)\,u +\nabla\,p+a\,|u|^{2\alpha}u+b\,|u|u=f,\]
are studied in \cite{2,3,12,Louaked,18,21} and references therein. The long time behavior of solutions and the existence of global attractor to \eqref{oursystem} has been studied in \cite{17,23,24,YCL} for very restrictive values and ranges of parameters $\alpha$ and $\beta$. Also, existence, decay rates and some qualitative properties of weak solutions are shown in \cite{Oliveira}. Recently, we also become aware of the works \cite{Cai,Varga}. In \cite{Cai}, the system \eqref{oursystem} is shown to be well-posed (with $b=0$) in the whole space for a smaller range of powers $\alpha$ than the our. Our result, with periodic boundary conditions, obviously generalizes their result. In \cite {Varga}, the system \eqref{oursystem} is investigated. The authors show existence and uniqueness of solutions for all $\alpha>1$, with Dirichlet boundary conditions and regular enough initial data. Their argument relies on the maximal regularity estimate for the corresponding semi-linear stationary Stokes problem proved using some modification of the nonlinear localization technique.
\vskip6pt
\noindent
Next, we introduce a few commonly used function spaces. By abuse of notation, $\mathcal V$ will refer in the case of Dirichlet boundary conditions \eqref{dirichlet}, to the space  $ \{ u\in \mathcal C_c^{\infty}(\Omega):\; \nabla\cdot u=0\}$ and to $\{u\,\text{is a trigonometric polynomial }:\; \nabla\cdot u=0\}$, in the case of periodic boundary conditions  \eqref{pbc}. Furthermore, we introduce
\begin{align*}
&{\bf H}:={\rm closure \;of}\;\mathcal{V}\;{\rm in}\;L^2(\Omega) \quad \text{and}\quad  {\bf V}:={\rm closure \;of}\;\mathcal{V}\;{\rm in}\;\1.
\end{align*}
The space $\mathbf H$ is endowed with the scalar product, $\langle\cdot,\cdot\rangle_{\bf H}$ induced by $L^2(\Omega)$. The Hilbert space $\mathbf V$ is equipped with the scalar product $\langle u,v\rangle_{\bf V} =\sum_{i=1}^3 \langle D_i\,u,D_i\,v\rangle_{{\bf H}}$ in the Dirichlet case, and  $\langle u,v\rangle_{\bf V} =\langle u,v\rangle_{\bf H}+\sum_{i=1}^3 \langle D_i\,u,D_i\,v\rangle_{{\bf H}}$ in the periodic case. In particular, it is well known that $ \bf V\subset H\equiv H'\subset V'$, with dense inclusions and continuous injections, see, {\it e.g.}, \cite{Constantin,Temam}. From now on we will use the notation $\nr \cdot\nr_p$ for all $p\geq1$ instead of $\nr\cdot\nr_{L^p(\Omega)} $  and $\nr\cdot\nr_{\infty,2}$ instead of $\nr \cdot\nr_{L^\infty(\mathbb R^+,L^2(\Omega))}$ for lightness of the notation. In the sequel, several inequalities will involve $\epsilon_i$ for $i=0,1,\dots$ originating from the application of Young's inequality. These inequalities will be valid for all $\epsilon, \epsilon_i>0$ for all $i=0,1,\ldots$ and we will omit mentioning this. Eventually, let us recall the Lady\v{z}henskaya, Agmon's, and Sobolev inequalities (see, {\it e.g.}, \cite{Constantin,Lady,Temam,Temam1})
\begin{align*}
&\nr u\nr_4 \leq \kappa_1 \,\nr u\nr_2^\frac14\,\nr u\nr_{H^1}^\frac34, \quad \text{for all}\quad u \in H^1(\Omega),\\
&\nr u\nr_{\infty} \leq  \kappa_2\,\nr u\nr^\frac12_{H^1}\,\nr  u\nr^\frac12_{H^2},\quad \text{ for all }\quad u\in  H^2(\Omega),\\
&\nr u\nr_6 \leq \kappa_3\,\nr u \nr_{H^1}, \quad \text{for all}\quad u \in H^1(\Omega),
\end{align*}
where $\kappa_1,\kappa_2,\kappa_3>0$, denote dimensionless scale invariant constants that depend only on the shape of the domain $\Omega$. These inequalities will be used tacitly in the estimates we will establish in this paper.
\vskip6pt
\noindent Our starting point is the existence and uniqueness of weak solutions to system \eqref{oursystem},\eqref{pbc}. More precisely, we have the following 
\vskip6pt
\begin{thm} \label{thmweak}
Let $f\in L^\infty(\mathbb R^+;{\bf H})$ and $u_0\in {\bf H}$. Assume $\alpha>\beta\geq 0,\,a>0$ and $b\in\mathbb{R}$, then systems  \eqref{oursystem},\eqref{pbc} and  \eqref{oursystem},\eqref{dirichlet} have a weak solution satisfying
\begin{align*}
&u \in C^0(\mathbb R^+;{\bf H}_{\rm weak})\cap L^\infty_{\rm loc}(\mathbb R^+; {\bf H})\cap L^2_{\rm loc}(\mathbb R^+; {\bf V}) \cap L^{{2\alpha+2}}_{\rm loc}(\mathbb R^+;L^{{2\alpha+2}}(\Omega)),\\&{\rm and}\\
&\limsup_{t\rightarrow +\infty}\,\nr u(t)\nr_{{\0}}\leq \rho_0,
\end{align*}
where
\[
\rho_0= \left[\frac{1}{2\nu}\,\left(\nr f\nr^2_{\infty,2} + \left\lbrace \left[\frac{2(1+\car)}{a}\right]^\frac1\alpha+ \car \,\left[\frac{4}{a}\right]^\frac{\beta+1}{\alpha-\beta}+2\,\left[ \frac\nu a\right]^{\frac{\alpha+1}{\alpha}}\right\rbrace\,|\Omega|\right)\right]^\frac12,\]
and $\car$ stands for the characteristic function of the set $\{b<0\}$. Moreover, if $2\alpha\geq3$, then the weak solutions depend continuously on the initial data, in particular they are unique.
\end{thm}
\noindent This Theorem ensures the existence and uniqueness of weak solutions to system \eqref{oursystem} with Dirichlet type boundary condition \eqref{dirichlet} or periodic boundary conditions \eqref{pbc}  if $\alpha\geq \frac32$ . This result was already obtained in \cite{Oliveira, Varga}. It is rather easy to extend the existence to the range $\alpha>-\frac12$ but not the uniqueness. Furthermore, one can also show  the so-called {\it finite-time extinction} of the solutions if $-\frac12<\alpha<\frac12$, using Gagliardo-Nirenberg-Sobolev inequality and we refer the reader to \cite{Oliveira}. Let us also mention that the assumption $f\in L^\infty(\mathbb R^+;{\bf H})$ can be weakened by assuming only the local integrability of time of  $\nr f\nr^2_{\bf H}$.
\vskip6pt
\noindent For more regular solutions of system \eqref{oursystem},\eqref{pbc} we have:
\vskip6pt
\begin{thm}\label{thm}
Let $f\in L^\infty(\mathbb R^+;{\bf H})$ and $u_0\in {\bf V}$. Assume $\alpha>1, 0\leq \beta<\alpha , a>0$ and $b\in\mathbb{R}$. Then system  \eqref{oursystem},\eqref{pbc} has a unique global strong solution satisfying
\begin{align*}
&u \in C^0_{\rm b}(\mathbb R^+;{\bf V}) \cap L^2_{\rm loc}(\mathbb R^+; H^2(\Omega) \cap {\bf V}) \cap L^{{2\alpha+2}}_{\rm loc}(\mathbb R^+;L^{{2\alpha+2}}(\Omega)),
\end{align*}
and $\limsup_{t\rightarrow +\infty}\,\nr \nabla u(t)\nr_{2}\leq \rho_1$, where 
\begin{align*}
&\rho_1:=\left[ \frac{1}{\nu}\, \left\lbrace \left( 2+(\mathcal A_1+1)\,\left[1+\frac1\nu\right]\right)\,\nr f\nr^2_{\infty,2}+ (\mathcal A_1+1)\,\left[\eta_0+\frac{\eta_1}{\nu}\right]  \right\rbrace\right]^\frac12,
\end{align*}
with
\[\mathcal A_1:= 2\left\lbrace \left[\frac{\nu^\alpha(1+2\alpha)}{2(1+\car)}\right]^{\frac1{1-\alpha}} + \car\left[\frac{4^\beta[|b|(1+2\beta)]^\alpha}{[a(1+2\alpha)]^\beta}\right]^{\frac1{\alpha-\beta}} \right\rbrace,\]
and
\[\eta_0= \left\lbrace \left[\frac{2(1+\car)}{a}\right]^\frac1\alpha+ \car \,|b|\,\left[\frac{4|b|}{a}\right]^\frac{\beta+1}{\alpha-\beta}\right\rbrace\,|\Omega|,\quad\eta_1:= \eta_0 + \frac a2\,\left[ \frac\nu a\right]^{\frac{\alpha+1}{\alpha}}\,|\Omega|.\]
Furthermore, the solution depends continuously on the initial data. If in addition, if $u_0\in L^{2\alpha+2}(\Omega)$ then $u\in L^{\infty}_{\rm loc}(\mathbb R^+;L^{{2\alpha+2}}(\Omega))$ and $\partial_t u \in L^2_{\rm loc}(\mathbb R^+;L^2(\Omega))$.
\end{thm}
\noindent In particular, Theorem \ref{thm} ensures the global well-posedness of the traditional 3D BFeD system with periodic boundary conditions. Let us mention that we use here periodic boundary conditions instead of the Dirichlet ones to avoid technicalities and present a clear picture of the data assimilation algorithm (see below). However, our results remain valid with Dirichlet boundary condition as it will be done in a follow-up paper.  In addition, we observe that these strong solutions are unique within the class of weak solutions as it will be described below. Observe, however, that as a consequence of Theorem \ref{thmweak} it follows that system \eqref{oursystem},\eqref{pbc} has an absorbing ball in $L^2(\Omega)$, and by the regularity properties guaranteed by Theorem \ref{thm} (for $\alpha>1$) one is able to show that  our system has a  finite-dimensional global attractor, provided $f\in \bf H$ is time independent (cf. \cite{Constantin,Temam1}).
\begin{remark}
Let us mention that in the paper \cite{Varga}, there is no result concerning global existence and uniqueness of strong solutions with initial data in $H^1$. However, using the maximal regularity technique, the authors showed global existence and uniqueness for initial data in $ u_0\in H^2$. This implies in turn that the solution belongs to $L^\infty (\mathbb R^+; L^\infty(\Omega))$ which yields simpler estimates concerning the nonlinear terms.
\end{remark}

\vskip6pt
\noindent  Theorem \ref{thm} and Theorem \ref{thmweak} show that the nonlinear damping term, when $a>0$, dominates the pumping term for all $b\in \mathbb R$. However, when $a<0$, we can show that the pumping term may dominate in some particular situations, for large enough initial data, a finite-time blow-up of the solutions occurs. In other words,  we show that it is possible to construct a finite-time blowing up solutions when $a<0$ regardless the value of $b$. For that purpose, we consider system \eqref{oursystem} in the  periodic channel
\begin{equation}\label{perdomain}
\Omega=\left\lbrace(x,y,z)\;:\quad x\in \mathbb T,\quad y\in \mathbb T,\quad 0\leq z\leq L\right\rbrace,
\end{equation}
for a given $L>0$, subject to the mixed periodic and Dirichlet boundary conditions
\begin{equation}\label{bc}
\left\lbrace
\begin{array}{ll}
&u(x,y,z,t)=u(x+L,y,z,t)=u(x,y+L,z,t)=u(x+L,y+L,z,t),\\&\\
&u(x,y,0,t)=u(x,y,L,t)=0,\quad{\rm for \:\:all}\quad x,y\in \mathbb{T} \quad {\rm and}\quad z\in[0,L].
\end{array}
\right.
\end{equation}
\vskip6pt
\noindent Then, we have the following:
\vskip6pt
\begin{thm}\label{blowupsol}
Let $\Omega$ be as in \eqref{perdomain}, $a<0$ and $b\in\mathbb{R}$. Assume that $\alpha>\beta\geq0$. Then, there exists an initial data $u_0$ for which the corresponding solution of system \eqref{oursystem},\eqref{bc} blows up in finite time.
\end{thm}
\vskip6pt
\noindent
Basically, the idea of data assimilation is to incorporate spatially discrete measurements and observations into the physical or numerical model. The main motivation of data assimilation is that deterministic models are imperfect, most of the time incomplete because of the inability to account for all relevant process.  In particular, it has played a central role in the improvement of weather forecast and there is a growing interest in applying data assimilation to problems arising from several applications. At the mathematical level and in the case of our model \eqref{oursystem}, the method of continuous data assimilation was introduced first in \cite{azouani} in the context of the 2D Navier-Stokes equations. A follow-up work was done in \cite{EdrissDeb} and \cite{Aseel} (see also references therein, in particular \cite{BOT} for the case of stochastically noisy data). The algorithm can be described as follows. Assume the initial data $u_0$ of system \eqref{oursystem},\eqref{pbc} is missing. Also, assume that discret spatial observations of $u(t)$ can be used to construct an interpolation operator $\mathcal I_h(u)$ where $h$ denotes spatial  resolution of the collected measurements. For instance, $\mathcal I_h(u)$ can represent an interpolant operator based on spatial observations of the reference solution of \eqref{oursystem} at a coarse spatial resolution of size $h$. Now, Let $\mu>0$ be a nudging parameter, to be fixed later on, and consider the following data assimilation algorithm
\begin{align}{\label{oursystemu}}
\left\lbrace\begin{array}{ll}
&\partial_t\,v-\nu\,\Delta\,v +(v\cdot\nabla)\,v +\nabla\,q+a\,|v|^{2\alpha}\,v+b\,|v|^{2\beta}\,v=f+ \mu\,\left(\mathcal I_h(u)-\mathcal I_h(v)\right),\\ \\
& \nabla\cdot v=0,\;v\vert_{t=0}=v_0,
\end{array}
\right.
\end{align}
where $v_0$ is arbitrary (to be specified in the statement of the Theorems) with the periodic boundary conditions
\begin{equation}\label{pbcmu} \left\lbrace\begin{array}{lcl}
&v(x+L,y,z,t)=v(x,y+L,z,t)=v(x,y,z+L,t) =v(x,y,z,t),\\ &\\
&q(x+L,y,z,t)=q(x,y+L,z,t)=p(x,y,z+L,t) =q(x,y,z,t),\end{array}\right.
\end{equation}
Then, we can obtain the original unknown solution $u(t)$ of system \eqref{oursystem}, asymptotically in time, as the limit of $v(t)$, the solution of \eqref{oursystemu}, provided the nudging parameter $\mu$ is large enough and the resolution parameter $h$ of the measurements of $u(t)$ is small enough. Furthermore, the convergence holds at an exponential rate.
\vskip6pt
\noindent In this paper, we will consider the following  types of interpolant operators. The first is a linear operator that approximates identity, $\mathcal I^1_h:\, \1\rightarrow L^2(\Omega)$, such that
\begin{equation}\label{b1}\nr \psi-\mathcal I^1_h(\psi)\nr^2_\0 \leq c_0\,h^2\,\nr \nabla\,\psi\nr^2_{\0},\end{equation}
for every $\psi\in H^1(\Omega)$, where $c_0>0$ is a dimensionless constant. The second interpolant operator, $\mathcal I^2_h:\, H^2(\Omega)\rightarrow L^2(\Omega)$, is  also a linear operator that approximates identity such that  
\begin{equation}\label{b2}\nr \psi-\mathcal I^2_h(\psi)\nr^2_\0 \leq c_0\,h^2\,\nr \nabla\,\psi\nr^2_{\0}+c_1\,h^4\,\nr \Delta\,\psi\nr^2_{\0},\end{equation}
for every $\psi \in H^2(\Omega)$, where $c_0>0$ and $c_1>0$ are dimensionless constants.
\vskip6pt \noindent Before we state our results concerning system \eqref{oursystemu},\eqref{pbcmu}, let us give examples of such interpolant operators. For instance, let $\phi_k(x)=\frac1{L^3}\,e^{\frac{2\pi i}{L^3} k\cdot x},$ for $k \in \mathbb Z^3$. The projector onto the low Fourier modes, up to wave numbers $|k|$ such that $|k| \leq \lfloor \frac{L}{2\pi h}\rfloor$ defined as
\begin{align*}
\mathcal I_h^1(\psi)(x) =\mathbf P_k\psi = \sum_{|k| \leq \lfloor \frac{L}{2\pi h}\rfloor} \hat{\psi}_k\,\phi_k(x),
\end{align*}
where $\psi(x) = \sum_{k\in \mathbb Z^3} \,\hat{\psi}_k\,\phi_k(x)$, is an interpolant of the first type. We leave the exercise of showing that this interpolant satisfies the functional inequality $ \nr \psi-\mathcal I^1_h(\psi)\nr^2_\0 \leq c_0\,h^2\,\nr \nabla\,\psi\nr^2_{\0} $ to the reader. The volume element operator is also an interpolant of the first kind. More precisely, if we divide $\Omega$ in cubes $\Omega_k$, for all $k=1,\ldots, N$,  with edge $ L\,N^{-\frac13}$ so that $|\Omega_k|= N^{-1}\,L^3$, then
\begin{align*}
\mathcal I^1_h(\psi)(x)= \sum_{k=1}^N\, \frac{ \mathds 1_{\{\Omega_k\}}(x) }{|\Omega_k|}\,\int_{\Omega_k}\,\psi(y)\,dy,\quad h= L\,N^{-\frac13}.
\end{align*}
That is, we consider that the average of the function $\psi$ on each cube $\Omega_k$ is given, and we refer to \cite{EdrissDeb,azouani} for a proof of \eqref{b1} when $\psi \in H^1(\Omega)$. Eventually, the interpolant operator obtained by using measurement at a discrete set of nodal points in $\Omega=[0,L]^3$. Let $x_i \in \Omega_i$, where $\Omega_i$ denotes the $i^{th}$ cube introduce above, be the points where the measurement of the velocity of the flow is taken. Therefore, the nodal point interpolant operator can be defined as
\[ \mathcal I_h^2(\psi(x))= \sum_{k=1}^N\, \psi(x_k)\,\mathds 1_{\{\Omega_k\}}(x).  \]
We refer the reader to \cite{azouani,EdrissDeb} for a proof showing that $\mathcal I_h^2$ satisfies $\nr \psi-\mathcal I^2_h(\psi)\nr^2_\0 \leq c_0\,h^2\,\nr \nabla\,\psi\nr^2_{\0}+c_1\,h^4\,\nr \Delta\,\psi\nr^2_{\0}$.
\vskip6pt
\noindent
Now, we turn to the statement of our results about system \eqref{oursystemu},\eqref{pbcmu}. Concerning the first interpolant, we have
\begin{thm}\label{thmu}
Let $f\in L^\infty(\mathbb R^+;{\bf H})$ and $v_0\in {\bf V}$. Assume $\alpha>1, 0\leq \beta <\alpha,\, a>0$, $b\in\mathbb{R}$, and $2 \mu \,c_0\,h^2 \leq \nu$. Then system  (\ref{oursystemu}-\ref{pbcmu}) with interpolant $\mathcal I^1_h$ has a unique global strong solution satisfying
\begin{align*}
&v \in C^0_{\rm b}(\mathbb R^+;{\bf V}) \cap L^2_{\rm loc}(\mathbb R^+; H^2(\Omega) \cap {\bf V}) \cap L^{{2\alpha+2}}_{\rm loc}(\mathbb R^+;L^{{2\alpha+2}}(\Omega)).
\end{align*}
Moreover, if $v_0\in L^{2\alpha+2}(\Omega)$, then $v\in  L^{\infty}_{\rm loc}(\mathbb R^+;L^{{2\alpha+2}}(\Omega))$ and $\partial_t v \in L^2_{\rm loc}(\mathbb R^+;L^2(\Omega))$. Furthermore, the solution depends continuously on the initial data $v_0$ in ${\bf V}$ norm.
\end{thm}
\noindent This Theorem ensures the well-posdness of the data assimilation algorithm when an interpolant of the first kind is used for all relaxation (nudging) parameter $\mu>0$ provided that $h$ is small enough. Furthermore, if $\mu$ is large enough, then we have
\begin{thm}\label{thmconI1}
Let $f\in L^\infty(\mathbb R^+;{\bf H})$, $v_0\in {\bf V}$, $\alpha>1, 0\leq \beta<\alpha, a>0$, and $b\in\mathbb{R}$. Let $\mu$ large enough such that
\[\mu>2\left\lbrace \frac12\,\nu+\frac{2^7\kappa_1^8}{\nu^3}K^4 +\,\car \left[\frac{(2\tilde\kappa_0\,|b|)^\alpha}{(a\,\kappa_0)^\beta}\right]^{\frac1{\alpha-\beta}} \right\rbrace,\]
where $K$ is given in Proposition \ref{hereisthebound} below. Let $h$ be small enough such that $2\mu \,c_0\,h^2 \leq \nu$, and  $u(t)$ be the global unique strong solution of system \eqref{oursystem},\eqref{pbc} given by Theorem \ref{thm}. Then the global unique strong solution of system \eqref{oursystemu},\eqref{pbcmu}, with  interpolant $\mathcal I_h=\mathcal I^1_h$, given by Theorem \ref{thmu} satisfies $\nr u(t)-v(t) \nr_\0 \to 0$, as $t\to +\infty$, at an exponential rate. Furthermore, if  $u_0,v_0\in {\bf V}$ satisfying $\nr v_0\nr_{H^1}\leq \tilde K$, $\nr u_0\nr_{H^1}\leq \tilde K$ with $\tilde K> \sqrt{\rho_0^2+\rho_1^2}$, $1<\alpha<2$, and $\mu$ is large enough so that \eqref{cond132} or \eqref{cond322} holds if $1<\alpha\leq \frac32$ or $\frac32<\alpha<2$, respectively, then $\nr u(t)-v(t) \nr_{H^1(\Omega)} \to 0$, as $t\to +\infty$, at an exponential rate
\end{thm}
\noindent Theorem \ref{thmconI1} shows that our algorithm converges toward the solution of the original system \eqref{oursystem},\eqref{pbc} in $L^2$ norm, Moreover,  Theorem \ref{thmconI1} shows the convergence in the $H^1$ norm if $1<\alpha<2$ for restricted initial data $v_0$ to be inside the absorbing ball of the solution of the original system. Notice that the restriction on the initial data $v_0$ does not limit the applicability of our algorithm. Indeed, $\tilde K$ depends on the size of the absorbing ball of $u(t)$, $\sqrt{\rho_0^2+\rho_1^2}$, which is given {\it a priori}. This restriction is a consequence of the fact that we are enable to control uniformly with respect to $\mu$ the $H^1$ norm of $v(t)$ for all initial data $v_0$. When, the second interpolant is used, we have
\begin{thm}\label{thmu2}
Let $f\in L^\infty(\mathbb R^+;{\bf H})$ and let $u_0, v_0\in {\bf V}$ such that $\nr v_0\nr_{H^1}\leq \tilde K$, $\nr u_0\nr_{H^1}\leq \tilde K$, with $\tilde K >\sqrt{\rho_0^2+\rho_1^2}$. Assume $1<\alpha<2, 0\leq \beta<\alpha, a>0$, and $b\in\mathbb{R}$. Let $u(t)$ be the global unique strong solution of system \eqref{oursystem},\eqref{pbc} given by Theorem \ref{thm}. Let $\mu$ be large enough such that \eqref{toputinthm1} holds if $1<\alpha<\frac32$, or \eqref{toputinthm2} holds if $\frac32\leq \alpha<2$, and $h$ be small enough such that $15\mu\max\{c_0,\sqrt {c_1}\}\,h^2 \leq \nu$. Then system  (\ref{oursystemu}-\ref{pbcmu}), with interpolant operator $\mathcal I_h=\mathcal I^2_h$ has a global unique strong solution satisfying
\begin{align*}
&v \in C^0_{\rm b}(\mathbb R^+;{\bf V}) \cap L^2_{\rm loc}(\mathbb R^+; H^2(\Omega) \cap {\bf V}) \cap L^{\infty}_{\rm loc}(\mathbb R^+;L^{{2\alpha+2}}(\Omega)),\\&{\rm and}\\
&\partial_t v \in L^2_{\rm loc}(\mathbb R^+;L^2(\Omega)).
\end{align*}
Furthermore, the solution depends continuously on the initial data and satisfies $\nr u(t)-v(t) \nr_{H^1(\Omega)} \to 0$, as $t\to +\infty$, at an exponential rate.
\end{thm}
\noindent Observe that in this Theorem, we do not assume the initial data $v_0$ to be in ${\bf V} \cap L^{2\alpha+2}(\Omega)$ as in Theorem \ref{thmu}. This is because we restrict the range of $\alpha$ to $1<\alpha<2$, so that by Sobolev embedding $v_0 \in {\bf V}$ implies that  $L^{2\alpha+2}(\Omega)$.
\section{Proof of Theorems \ref{thmweak} and \ref{thm}} \label{sectionthm12}
\noindent In this section, we prove Theorem \ref{thmweak} and Theorem \ref{thm}.
The global (in time) existence of solutions is shown in a classical way by proceeding in three steps. First, we use a Faedo-Galerkin approximation procedure to show the short time existence of solutions. The reader is referred to, for example, \cite{Constantin,Temam,Temam1} for details. Next, we obtain the necessary{\it a priori} bounds that allow to extend the solution of the Faedo-Galerkin system  globally in time. Eventually, we pass to the limit in the approximation procedure using the Aubin compactness Theorem, relying on the established {\it a priori} bounds.
\vskip6pt
\subsection{Existence of local solutions:}\label{existence galerkin}
Let  $\mathcal W_k(x)$ be the eigenfunctions of the Stokes operator, corresponding to the eigenvalues $\lambda_k$, for $k=1,2, \ldots$, repeated according to their multiplicities such that $0<\lambda_1\leq\lambda_2 \leq \cdots \leq \lambda_n\leq \lambda_{n+1}\leq \cdots$ and  $\mathbf V_m:={\rm Span}\left\lbrace\mathcal W_1,\ldots,\mathcal W_m\right\rbrace$. It is well known that the family $\{\mathcal W_k\}_{k=1}^\infty$ forms an orthonormal basis of ${\bf H}$ (see \cite{Constantin,Fois,Temam,Temam1}).  Now, let $T>0$, be an arbitrary time, and  consider the Faedo-Galerkin approximate function
\begin{equation}\label{eqpart0}{\bf u}_m(x,t)=\sum_{k=1}^m g_m^i(t)\,\mathcal W_i(x) \in \mathbf V_m.\end{equation}
Using system \eqref{oursystem},\eqref{pbc}, the unknown coefficients $g_m^i= \int_\Omega \mathbf u_m\cdot \mathcal W_i\,dx$, for $i=1,2, \ldots, m$ solve  the following system of ordinary differential equations in $\mathbf V_m$
\begin{align}
&\frac d{dt}\,\int_\Omega \mathbf u_m\cdot \mathcal W_k\,dx+\nu\,\int_\Omega \nabla\mathbf u_m\cdot \nabla\mathcal W_k\,dx + \int_\Omega (\mathbf u_m\cdot\nabla)\mathbf u_m\cdot \mathcal W_k\,dx \nonumber \\ &\hskip30pt+ a\,\int_\Omega |\mathbf u_m|^{2\alpha} \mathbf u_m\cdot\mathcal W_k\,\,dx +b\int_\Omega |\mathbf u_m|^{2\beta} \mathbf u_m\cdot\mathcal W_k\,dx = \int_\Omega f\cdot\mathcal W_k\,dx,\label{eqpart1}\\
\label{eqpart2}  &g_m^k(t=0)=\int {\mathbf u_0}\cdot \mathcal W_k\,dx,
\end{align}
for all $k=1,\ldots,m$. Since the vector field in system \eqref{eqpart0},\eqref{eqpart2} is locally Lipschitz, the system admits a unique solution $g_m^k(t)\in C^1([0,T_m])$, for some small interval $[0,T_m]\subset [0,T]$.
\vskip6pt
\noindent The solutions of  system \eqref{oursystem},\eqref{pbc} are obtained as the limit of a subsequence $u_{m_j}$  (as $j\to \infty$) of the  sequence of solutions $u_m$ of the Faedo-Galerkin system \eqref{eqpart0},\eqref{eqpart2}. Specifically, this will be achieved   by the means of the Aubin compactness Theorem, and the corresponding bounds are reached by Banach-Alaoglu Theorem. Also, let us mention that for clarity of the presentation, we will not introduce explicitly the Leray projector and the Stokes operator in (\ref{oursystem}-\ref{pbc}), which is usually used in the analysis of the Navier-Stokes equation to get rid of the pressure. We will also omit the exercise of passing  to the limit to the reader in (\ref{eqpart0}-\ref{eqpart2}) using the{\it a priori} bounds.
\vskip6pt
\subsection{Energy Estimates:}
\noindent In this paragraph, we show that $u(t)$ is uniformly bounded in $L^2(\Omega)$ with respect to time, $\nabla u \in L^2(0,T; L^2(\Omega))$ and $u \in L^{{2\alpha+2}}(0,T; L^2(\Omega))$. Let us mention that since the average of $u$ is not zero, we cannot rely on Poincar\'e inequality, as it is usually used  in the theory of Navier-Stokes. Instead, we rely strongly on the damping nonlinear term as it will be clear in the sequel. For this purpose, we multiply the first equation of (\ref{oursystem}) by $u$ and integrate over $\Omega$. Integrating by parts, we obtain
\begin{align}\label{dash1}
\frac12\,\frac{d}{dt}\, \nr u \nr^2_2+\nu\,\nr \nabla u\nr^2_2 +a\,\nr u\nr^{2\alpha+2}_{2\alpha+2}+b\,\nr u\nr^{2 \beta+2}_{2\beta+2} &= \s\,f(t,x)\,\cdot\,u(t,x)\,dx.
\end{align}
Since $0<\beta<\alpha$, we use H\"older and Young inequalities to get
\begin{align}
\label{new0} \nr u\nr_2^2 &\leq \epsilon\,\nr u\nr_{{2\alpha+2}}^{2\alpha+2} +  \left[\frac1\epsilon\right]^\frac1\alpha\,|\Omega|,\\
\label{relationalphabeta}\nr u\nr_{{2\beta+2}}^{2\beta+2} &\leq \epsilon\,\nr u\nr_{{2\alpha+2}}^{2\alpha+2} + \left[\frac1\epsilon\right]^{{\frac{\beta+1}{\alpha-\beta}}}\,|\Omega|,\\
\label{new} \s f\cdot u \,dx &\leq  \epsilon\,\nr u\nr^{2\alpha+2}_{2\alpha+2} + \left[\frac1\epsilon\right]^{\frac1{2\alpha+1}}\,\nr f\nr^{1+\frac1{2\alpha+1}}_{1+\frac1{2\alpha+1}} \nonumber\\
&\leq  \epsilon\,\nr u\nr^{2\alpha+2}_{2\alpha+2} + \,\nr f\nr^{2}_2 +  \left[\frac1{\epsilon}\right]^{\frac{1}{\alpha}}\,|\Omega|.
\end{align}
Plugging \eqref{relationalphabeta} and \eqref{new} in \eqref{dash1} and optimizing in $\epsilon$ leads to
\begin{align}\label{dash2}
&\frac{d}{dt}\nr u\nr^2_{\0} +2\nu\,\nr\nabla\,u\nr^2_{\0} +a\,\nr u\nr^{{2\alpha+2}}_{{{2\alpha+2}}}\leq 2\left(\nr f\nr^2_{\0} +\eta_0\right),
\end{align}
where we set
\[\eta_0:= \left\lbrace \left[\frac{2(1+\car)}{a}\right]^\frac1\alpha+ \car \,|b|\,\left[\frac{4|b|}{a}\right]^\frac{\beta+1}{\alpha-\beta}\right\rbrace\,|\Omega|.\]
Thanks to \eqref{new0}, we get (for a different choice of $\epsilon$)
\begin{align}\label{dash2new}
\frac{d}{dt}\nr u\nr^2_{\0} +2\nu\,\nr\nabla\,u\nr^2_{\0} &+\nu\,\nr u\nr^2_2\leq {2}\,\left(\nr f\nr^2_{\0}+{\eta_1}\right),
\end{align}
where
\[ \eta_1:= \eta_0 + \frac a2\,\left[ \frac\nu a\right]^{\frac{\alpha+1}{\alpha}}\,|\Omega|.\]
In particular,
\begin{align*}
\frac{d}{dt}\nr u\nr^2_{\0} +\nu\,\nr u\nr_2^2 \leq {2}\,\left(\nr f\nr^2_{\0}+{\eta_1}\right).
\end{align*}
Using Gronwall's inequality, we end up with the following uniform bound with respect to time
\begin{equation}\label{mass-bound-neg}
\nr u(t)\nr^2_{\0} \leq \nr u_0\nr^2_{\0}\,e^{-\nu\,t}+\frac{2}{\nu}\,\left(\nr f\nr^2_{\infty,2} + \eta_1\right).
\end{equation}
Now, integrating \eqref{dash2} with respect to time, we obtain
\begin{align}\label{gradient-bound00}
2\,\nu\,\st\nr\nabla\,u(s)\nr^2_{\0}\,ds &+a\,\st\,\nr u(s)\nr^{{2\alpha+2}}_{{{2\alpha+2}}(\Omega)}\,ds \leq  \nr u_0\nr^2_{\0}+{2}\,\left(\nr f\nr^2_{\infty,2} + \eta_0\right)\,t.
\end{align}
 Obviously, this inequality gives the required estimates for the existence of weak solutions for system \eqref{oursystem},\eqref{pbc}. Furthermore, \eqref{mass-bound-neg} shows that
 \[\limsup_{t\to +\infty}\,\nr u(t)\nr_2^2 \leq \frac{2}{\nu}\,\left(\nr f\nr^2_{\infty,2} + \eta_1\right).\]
\vskip6pt
\subsection{Gradient Estimate:}\label{gradientparagraph}
Next, we show that $\nabla u \in L^2(0,T; L^2(\Omega))$ and that $\Delta u \in L^2(0,T; L^2(\Omega))$. For this purpose, we multiply the first equation of (\ref{oursystem}) by $-\Delta u$ and integrate over $\Omega$ to get
\begin{align}\label{energy-est1}
\nonumber\frac12\,\frac{d}{dt}\,\nr \nabla u\nr^2_2 +\nu\,\nr \Delta u\nr^2_2+\s\,(u\cdot\nabla)\,u\,\cdot\,(-\Delta\,u)\,dx +a\,(1+2\alpha) \,\nr |u|^\alpha\,\nabla u\nr^2_2\\ +b\,(1+2\beta) \,\nr |u|^\beta\,\nabla u\nr^2_2 \leq \frac1\nu\,\nr f\nr^2_{\infty,2} + \frac\nu4\,\nr\Delta u\nr_2^2.
\end{align}
To obtain this inequality, we used the fact that $\s\,|u|^{2\alpha}\,u\,\cdot\,(-\Delta\,u)\,dx
=(1+2\alpha)\,\nr |u|^\alpha\,\nabla u\nr^2_2$. Now, using Cauchy-Schwarz, H\"older and Young inequalities and assuming $\alpha>1$, we can write
\begin{align}
&\left|\s\,(u\cdot\nabla)\,u\,\cdot\,(-\Delta\,u)\,dx \right| \leq \s\,|u|\,|\nabla\,u|^\frac1\alpha\,|\nabla\,u|^{1-\frac1\alpha}\,|\Delta\,u|\,dx \nonumber\\
&\hskip30pt\leq \nr |u|\,|\nabla\,u|^\frac1\alpha\nr_{L^{2\alpha}(\Omega)}\,\nr |\nabla\,u|^{1-\frac1\alpha}|\nr_{L^\frac{2\alpha}{\alpha-1}(\Omega)}\,\nr\Delta\,u\nr_{\0}\nonumber\\
&\hskip30pt\leq \frac1{4\epsilon_0}\, \nr|u|^{\alpha}\,|\nabla\,u|\nr^{\frac2\alpha}_{\0}\,\nr\nabla\,u\nr^{2(1-\frac1\alpha)}_{\0}  +\epsilon_0\nr\Delta\,u\nr^2_{\0}\nonumber\\
&\hskip30pt\leq \frac\epsilon{4\epsilon_0\,}\, \nr|u|^{\alpha}\,\nabla\,u\nr^{2}_{\0} + \frac{\epsilon^{\frac1{1-\alpha}}}{4\epsilon_0} \,\nr\nabla\,u\nr^{2}_{\0}  +\epsilon_0\,\nr\Delta\,u\nr^2_{\0}.\label{cumbersome}
\end{align}
In the sequel, we will need as well the following estimate
\begin{align}
\nr |u|^\beta\,\nabla u\nr_2^2 &\leq \nr |u|^\alpha\,\nabla u\nr_2^{\frac{2\beta}{\alpha}} \, \nr \nabla u\nr_2^{2\frac{\alpha-\beta}{\alpha}} \leq \epsilon \,\nr |u|^\alpha\,\nabla u\nr_2^2 + \left[ \frac1\epsilon \right]^{\frac\beta{\alpha-\beta}}\,\nr \nabla u\nr_2^2.\label{relationbad}
\end{align}
Gathering both estimates, we get
\begin{align}
\frac{d}{dt}\,\nr \nabla\,u\nr^2_{\0} &+\nu \nr \Delta\,u\nr^2_{\0} +  a\,(1+2\alpha)  \,\nr |u|^\alpha\,|\nabla\,u|\nr^2_{\0} \leq \frac2{\nu}\,\,\nr f\nr^2_{\0} +\mathcal A_1\,\nr\nabla\,u\nr^2_{\0},\label{touse0}
\end{align}
where we set
\[\mathcal A_1:= 2\left\lbrace \left[\frac{\nu^\alpha\,a\,(1+2\alpha)}{2(1+\car)}\right]^{\frac1{1-\alpha}} + \car\left[\frac{4^\beta[|b|(1+2\beta)]^\alpha}{[a(1+2\alpha)]^\beta}\right]^{\frac1{\alpha-\beta}} \right\rbrace.\]
In particular, we have
\begin{align}
\frac{d}{dt}\,\nr \nabla\,u\nr^2_{\0} &\leq \frac2{\nu}\,\,\nr f\nr^2_{\0} +\mathcal A_1\,\nr\nabla\,u\nr^2_{\0}.\label{tointegjusthere}
\end{align}
Integrating  \eqref{tointegjusthere} with respect to time, and using estimate (\ref{gradient-bound00}) leads to
\begin{align}\nr \nabla\,u(t)\nr^2_{\0}&\leq \frac1\nu\,\left[\left(2 + {\mathcal A_1}\right)\,\nr f\nr^2_{\infty,2} +{\mathcal A_1\,\eta_0} \right]\,t  + \frac{\mathcal A_1}{2\nu}\nr u_0\nr^2_{\0}+\nr \nabla\,u_0\nr^2_{\0}.
\label{boundtoimprove}
\end{align}
for all $t\in[0,T]$. Now, we integrate \eqref{touse0} with respect to time, and obtain
\begin{align}
\nu\,\int_0^t\,\nr \Delta\,u(s)\nr_\0^2\,ds &+ a\,(1+2\alpha)\,\int_0^t\,\nr |u(s)|^\alpha\,\nabla\,u(s)\nr_\0^2\,ds  \nonumber\\ &\leq  \frac1\nu\,\left[\left(2 + {\mathcal A_1}\right)\,\nr f\nr^2_{\infty,2} +{\mathcal A_1\,\eta_0} \right]\,t  + \frac{\mathcal A_1}{2\nu}\nr u_0\nr^2_{\0}+\nr \nabla\,u_0\nr^2_{\0}.
\label{toconcludepos}
\end{align}
\vskip6pt
\noindent In summary, \eqref{mass-bound-neg}, and \eqref{toconcludepos}) show that $u\in L^2_{\rm loc}(\mathbb R^+; H^2(\Omega))$. In particular, we conclude that $\Delta u\in L^2_{\rm loc}(\mathbb R^+; L^2(\Omega))$ and $u\in L^2_{\rm loc}(\mathbb R^+; L^\infty(\Omega))$, thanks to Sobolev inequality. Before we show the continuous dependence on the initial data, the uniqueness, and the fact that $\partial _t u\in L^2_{\rm loc}(\mathbb R^+; L^2(\Omega))$, we first improve the bound \eqref{toconcludepos}.
\vskip6pt
\subsection{Uniform Gradient Estimate:}
Next, we show that $\nabla u$ belongs to $L^\infty(\mathbb R^+; L^2(\Omega))$, that is $\nr \nabla u\nr_2$ is uniformly bounded with respect to time. More precisely, we claim
\begin{prop}\label{hereisthebound}
Let $f\in L^\infty(\mathbb R^+;{\bf H})$ and $u_0\in {\bf V}$. Assume $\alpha>1, 0\leq\beta<\alpha , a>0$, and $b\in\mathbb{R}$. Then, the solutions of (\ref{oursystem}-\ref{pbc}) satisfy
\begin{equation*}
\nr \nabla u(t)\nr_2 \leq K:= \left\lbrace
\begin{array}{lcl}
K_1& {\rm if} & 0\leq t\leq 1,\\
K_2& {\rm if} & t>1 ,
\end{array}
\right.
\end{equation*}
where $K_1$ and $K_2$ are given in \eqref{K1} and \eqref{K2}, below, and $\mathcal A_1$ is  defined in Theorem \ref{thm}.
\end{prop}
\begin{proof}
 We start with the bound \eqref{boundtoimprove}. On the one hand, for all $t\in[0,1]$, we have
\begin{align}
\nr \nabla\,u(t)\nr^2_{\0} \leq \frac1\nu\,\left[\left(2 + {\mathcal A_1}\right)\,\nr f\nr^2_{\infty,2} +{\mathcal A_1\,\eta_0} \right]  + \frac{\mathcal A_1}{2\nu}\nr u_0\nr^2_{\0}+\nr \nabla\,u_0\nr^2_{\0}.\label{est-part1}
\end{align}
On the other hand, for all $t> 1$, we consider $s \in[t-1,t]$, and integrate (\ref{tointegjusthere}) over $[s,t]$ to get
\begin{align}
\nr \nabla\,u(t)\nr^2_{\0} &\leq \nr \nabla\,u(s)\nr^2_{\0} +\frac{2}{\nu}\,\nr f\nr^2_{\infty,2}\,(t-s) + {\mathcal A}_1\,\int_s^t\nr\nabla\,u(\tau)\nr^2_{\0}\,d\tau\nonumber\\
&\leq  \nr \nabla\,u(s)\nr^2_{\0} +\frac{2}{\nu}\,\nr f\nr^2_{\infty,2} +\mathcal A_1\,\int_{t-1}^t\nr\nabla\,u(\tau)\nr^2_{\0}\,d\tau.\label{newdash1}
\end{align}
Notice that integrating (\ref{dash2}) over $[t-1,t]$, and using (\ref{mass-bound-neg}) we obtain
\begin{align}
&\nr u(t)\nr^2_{\0} +2\,\nu \int_{t-1}^t\nr\nabla\,u(s)\nr^2_{\0}\,ds \leq {2}\,\left(\nr f\nr^2_{\infty,2} +\eta_0\right)+ \nr u(t-1)\nr^2_{\0} \nonumber\\
&\hskip70pt\leq {2}\,\left\lbrace \left(1+\frac1\nu\right)\nr f\nr^2_{\infty,2} +\eta_0+\frac{\eta_1}{\nu}\right\rbrace+e^{-\nu\,(t-1)}\, \nr u_0\nr^2_{\0}.\label{e1}
\end{align}
Plugging (\ref{e1}) into \eqref{newdash1}, we get for all $s\in[t-1,t]$ such that $t> 1$,
\begin{align*}
\nr \nabla\,u(t)\nr^2_{\0} &\leq  \nr \nabla\,u(s)\nr^2_{\0} + \frac{\mathcal A_1}{2\nu}\,e^{-\nu(t-1)} \,\nr u_0\nr^2_{\0} \\
&+\frac{1}{\nu}\, \left\lbrace \left( 2+{\mathcal A_1}\left[1+\frac1\nu\right]\right)\,\nr f\nr^2_{\infty,2}+ {\mathcal A_1}\,\left[\eta_0+\frac{\eta_1}{\nu}\right]  \right\rbrace.
\end{align*}
Now, we integrate this inequality with respect to $s$ over $[t-1,t]$ and obtain
\begin{align}
 \nonumber\nr &\nabla\,u(t)\nr^2_{\0} \leq \int_{t-1}^t \nr \nabla\,u(s)\nr^2_{\0}\,ds +\frac{\mathcal A_1}{2\nu}\,e^{-\nu(t-1)} \,\nr u_0\nr^2_{\0}\\ &\hskip25pt+\frac{1}{\nu}\, \left\lbrace \left( 2+{\mathcal A_1}\left[1+\frac1\nu\right]\right)\,\nr f\nr^2_{\infty,2}+ {\mathcal A_1}\,\left[\eta_0+\frac{\eta_1}{\nu}\right]  \right\rbrace\nonumber\\
&\hskip25pt\leq \frac1{2\nu}\,\left(1+\mathcal A_1\right)\,e^{-\nu(t-1)}\,\nr u_0\nr_2^2 \nonumber \\
&\hskip25pt+\frac{1}{\nu}\, \left\lbrace \left( 2+(\mathcal A_1+1)\,\left[1+\frac1\nu\right]\right)\,\nr f\nr^2_{\infty,2}+ (\mathcal A_1+1)\,\left[\eta_0+\frac{\eta_1}{\nu}\right]  \right\rbrace.\label{e2}
\end{align}
Eventually, \eqref{est-part1} and \eqref{e2} show that
\[
\nr \nabla\,u(t)\nr_{\0} \leq K_1, \quad {\hbox {for all} } \quad 0 \leq t\leq 1,
\]
with
\begin{equation}\label{K1}K_1:=\left[  \frac1\nu\,\left[\left(2 + {\mathcal A_1}\right)\,\nr f\nr^2_{\infty,2} +{\mathcal A_1\,\eta_0} \right]  + \frac{\mathcal A_1}{2\nu}\nr u_0\nr^2_{\0}+\nr \nabla\,u_0\nr^2_{\0}\right]^\frac12,\end{equation}
and
\begin{equation*}
\nr \nabla\,u(t)\nr_{\0} \leq K_2, \quad {\hbox {for all} }\quad t>1,
\end{equation*}
with
\begin{align}K_2&:= \left[ \frac{1}{\nu}\, \left\lbrace \left( 2+(\mathcal A_1+1)\,\left[1+\frac1\nu\right]\right)\,\nr f\nr^2_{\infty,2}+ (\mathcal A_1+1)\,\left[\eta_0+\frac{\eta_1}{\nu}\right]  \right\rbrace\right.\nonumber\\
&\hskip205pt\left. +\frac1{2\nu}\,\left(1+\mathcal A_1\right)\,\nr u_0\nr_2^2\right]^\frac12.\label{K2}
\end{align}
\end{proof}
\noindent It is worth noticing that this uniform bound  is essential for establishing the existence of a nonempty compact global attractor of system \eqref{oursystem},\eqref{pbc} when $f$ is independent of time, and we refer to \cite{Varga}. As we will see below, this bound will be crucial in showing the convergence of the solution of the data assimilation algorithm to the solution of (\ref{oursystem}-\ref{pbc}). Eventually, observe that using \eqref{e2}, we obtain clearly the $\limsup$ bound of Theorem \ref{thm}.
\vskip6pt
\noindent Now, integrating (\ref{touse0}), leads to the following inequality for all $t\geq 0$
\begin{align}
\nonumber\nu \,\int_t^{t+1}\nr \Delta\,u(s)\nr^2_{\0}\,ds &+ {a\,(1+2\alpha)}\,\int_t^{t+1}\nr |u(s)|^\alpha\,|\nabla\,u(s)|\nr^2_{L^2}\,ds \\ &\leq\frac2{\nu}\,\nr f\nr^2_{\infty,2} + (\mathcal A_1+1) \,{K}^2.\label{mainone}
\end{align}
\vskip6pt
\noindent Now, we show that if the initial data $ u_0 \in {\bf V} \cap L^{2\alpha+2}(\Omega)$, then we have $u\in L^{\infty}_{\rm loc}(\mathbb R^+;L^{{2\alpha+2}}(\Omega))$ and $\partial _t u\in L^2_{\rm loc}(\mathbb R^+; L^2(\Omega))$. For this purpose, we multiply the first equation of (\ref{oursystem}) by $\partial_tu$ and integrate over $\Omega$ to get
\begin{align}
\nonumber\s\,|\partial_t u|^2\,dx + \s\,\partial_t u\,u\cdot \nabla u\,dx &+ \frac\nu2\,\frac{d}{dt}\s\,|\nabla u|^2\,dx +\frac{a}{2\alpha+2}\,\frac{d}{dt}\,\s|u^{2\alpha+2}\,dx \\ &+ \frac{b}{2\beta+2}\,\frac{d}{dt}\,\s|u^{2\beta+2}\,dx = \s f\,\partial_tu\,dx \label{multbydtu}
\end{align}
In particular, thanks to Cauchy-Schwarz and Young's inequalities, we can write
\begin{align}\label{themainfortimederivative}
\frac12\,\s\,|\partial_t u|^2\,dx + \s\,\partial_t u\,u\cdot \nabla u\,dx &+ \frac\nu2\,\frac{d}{dt}\s\,|\nabla u|^2\,dx +\frac{a}{2\alpha+2}\,\frac{d}{dt}\,\s|u^{2\alpha+2}\,dx \nonumber\\ &+ \frac{b}{2\beta+2}\,\frac{d}{dt}\,\s|u^{2\beta+2}\,dx \leq \nr f\nr_2^2
\end{align}
Furthermore, we have
\begin{align}\label{toseeitlater}
 &\s\,\partial_t u\,u\cdot \nabla u\,dx \leq \nr \partial_t u\nr_2\,\nr u\nr_\infty\,\nr \nabla u\nr_2 \leq \frac14\nr \partial_t u\nr^2_2 + K^2\nr u\nr^2_\infty\nonumber\\ &\hskip25pt\leq \frac14\nr \partial_t u\nr^2_2 + \left(\frac{\kappa_4^2K^4}{4\nu} +\nu\right)\,\nr u\nr^2_{H^1} + \nu\,\nr \Delta u\nr_2^2\nonumber \\
 &\hskip25pt\leq \frac14\nr \partial_t u\nr^2_2 + \left(\frac{\kappa_4^2K^4}{4\nu} +\nu\right) \left(\nr u_0\nr^2_{\0}+\frac{2}{\nu}\,\left(\nr f\nr^2_{\infty,2} + \eta_1\right) +K^2\right)+ \nu\,\nr \Delta u\nr_2^2\nonumber\\
 &\hskip25pt:= \frac14\nr \partial_t u\nr^2_2 + \nu\,\nr \Delta u\nr_2^2 +\rho_2
\end{align}
Thus, we can write now
\begin{align*}
 \frac14\nr \partial_t u\nr^2_2 + \frac\nu2\,\frac{d}{dt}\,\nr \nabla u\nr_2^2 + \frac{a}{2\alpha+2}\,\frac{d}{dt}\,\nr u\nr_{2\alpha+2}^{2\alpha+2}+ \frac{b}{2\beta+2}\,\frac{d}{dt}\,\nr u\nr_{2\beta+2}^{2\beta+2} \\ \leq   \nu\,\nr \Delta u\nr_2^2 + \nr f\nr_2^2+\rho_2.
\end{align*}
Integrating this inequality with respect to time and using \eqref{relationalphabeta} (if $b<0$) and \eqref{toconcludepos}, we obtain clearly
\[ u\in L^{\infty}_{\rm loc}(\mathbb R^+;L^{{2\alpha+2}}(\Omega)),\quad \partial _t u\in L^2_{\rm loc}(\mathbb R^+; L^2(\Omega)),\quad \text{for all}\quad u_0\in {\bf V}\cap L^{2\alpha+2}(\Omega).\]
\vskip6pt
\noindent Now we turn to the proof of the continuous dependence on initial data and the uniqueness.
\begin{remark} \label{Remreg}
As mentioned above, most of the calculation we presented are formal, but can be done rigorously by performing them at the Faedo-Galerkin level and passing to the limit. Nevertheless, the action of the equation on $u$ makes sense, in particular since we have clearly $\partial_t u \in L^2_{\rm loc} (\mathbb R^+; H^{-1}(\Omega))$ and $u \in L^2 (\mathbb R^+; H^{1}(\Omega))$ so that the function $t\mapsto \nr u(t)\nr_2^2$ is absolutely continuos and the action  $\left\langle  \partial_tu,u\right\rangle =\frac12\frac d{dt}\,\nr u(t)\nr_2^2$(see e.g. \cite{Temam,Temam1}). This operation is then rigorously defined when performed on the continuous equation. However, when we multiply by $-\Delta u$, the action $\left\langle  \partial_tu,-\Delta u\right\rangle $ is no longer equal to $\frac12\frac d{dt}\,\nr \nabla u(t)\nr_2^2$, except for initial data in $L^{2\alpha+2}(\Omega)$ where we show that $\partial _t u$ and $-\Delta u $ are both in $ L^2_{\rm loc}(\mathbb R^+; L^2(\Omega))$. Therefore, the multiplication by $-\Delta u$ and the integration over $\Omega$ has to be done either on the Faedo-Galerkin level, or using a regularization procedure like convolution with a mollifier, and  pass to the limit. Eventually, when the initial data is in ${\bf V} \cap L^{2\alpha+2(\Omega)}$, then $\partial_t u \in L^2_{\rm loc}(\mathbb R^+;L^2(\Omega))$ so that the multiplication of the equation by $\partial_t u $ makes sense. This remark holds true for all the systems we consider below.
\end{remark}

\subsection{Continuous dependence on initial data and uniqueness of strong solutions}
Let $u$ and $v$ be two solutions of system \eqref{oursystem},\eqref{pbc}, corresponding to initial data $u_0$ and $v_0$, respectively. Let $w=u-v$, then an easy calculation shows that $w$ satisfies
\begin{align*}
\partial_t\,w-\nu\,\Delta\,w +(w\cdot\nabla)\,u +(v\cdot\nabla)\,w +\nabla\,(p_u-p_v)+a\,\left(|u|^{2\alpha}u-|v|^{2\alpha}v\right)\\ +b\,\left(|u|^{2\beta}u-|v|^{2\beta}v\right)=0.
\end{align*}
The action of this equation on $w$ leads to \begin{align*}
\frac12\frac d{dt}\,\nr w (t)\nr^2_{\0} +\s [(w\cdot\nabla)\,u]\cdot w + \nu\,\nr \nabla\,w\nr^2_{\0} +a\,\s\left(|u|^{2\alpha}u-|v|^{2\alpha}v\right)\cdot w\,dx \\ +b\,\s\left(|u|^{2\beta}u-|v|^{2\beta}v\right)\cdot w\,dx=0.
\end{align*}
It is well known (see, {\it e.g.}, \cite{Barret}) that there exists a nonnegative constant $\kappa_0=\kappa_0(\alpha)$ such that
\begin{equation}
0\leq\kappa_0\,|u-v|^{2}\,\left(|u|+|v|\right)^{2\alpha}\leq \left(|u|^{2\alpha}u-|v|^{2\alpha}v\right)\cdot (u-v)
\label{positivity}.\end{equation}
Moreover, we use the following classical inequality
\begin{align}\label{classpower}
\left(|u|^{2\beta}u-|v|^{2\beta}v\right) \leq \tilde\kappa_0\,\left(|u|+|v|\right)^{2\beta} |w|.
\end{align}
Now, using H\"older, Young and  Lady\v{z}henskaya inequalities, we can write
\begin{align}
\nonumber\s [(w\cdot\nabla)\,u]\cdot w &\leq  \nr \nabla\,u\nr_{\0}\nr w\nr^2_{L^4(\Omega)} \\ & \leq \,\kappa^2_1\,\nr \nabla\,u\nr_{\0}\nr w\nr^\frac12_{\0}\,\nr w\nr^{\frac32}_{H^1} \nonumber\\ &\leq \epsilon \,\nr \nabla w\nr_2^2 + \left( \epsilon + \frac{\kappa^8_1}{\epsilon^3}\,\nr \nabla u\nr_2^4\right)\, \nr w\nr^2_{\0}.
\end{align}
Next, we use H\"older and Young inequalities to get
\begin{align}\label{diff}
\s\left(|u|^{2\beta}u-|v|^{2\beta}v\right)\cdot w\,dx& \leq \tilde\kappa_0\,\s\left(|u|+|v|\right)^{2\beta} \,|w|^{2\frac\beta\alpha}\,|w|^{2\left(1-\frac\beta\alpha\right)}dx\nonumber\\
&\leq \tilde\kappa_0\,\nr\left(|u|+|v|\right)^{\alpha} |w|\nr^{2\frac\beta\alpha}_{L^2(\Omega)}\, \nr w\nr^{2\frac{\alpha-\beta}{\alpha}}_{L^2(\Omega)}\nonumber\\
&\leq \epsilon\,\tilde\kappa_0\,\nr \left(|u|+|v|\right)^{\alpha} |w|\nr^2_{\0} +  \tilde\kappa_0\,\left[\frac1{\epsilon}\right]^{\frac\beta{\alpha-\beta}}\,\nr w\nr^2_{\0}.
\end{align}
Now, we use (\ref{positivity}--\ref{diff}), optimize in $\epsilon$, and obtain 
\begin{align*}
\frac d{dt}\,\nr w (t)\nr^2_{\0} +\nu\,\nr \nabla w\nr_2^2&+ (2-\car)\,{a\,\kappa_0},\nr \left(|u|+|v|\right)^{\alpha} |w|\nr^2_{\0}\\
& \leq \left\lbrace \nu+\frac{2^4\kappa_1^8}{\nu^3}K^4 +\car \left[\frac{(2\tilde\kappa_0\,|b|)^\alpha}{(a\,\kappa_0)^\beta}\right]^{\frac1{\alpha-\beta}}\right\rbrace\,\nr w(t)\nr^2_{\0},
\end{align*}
with $K$ as in Proposition \ref{hereisthebound}. In particular, it holds
\begin{align*}\frac d{dt}\,\nr w (t)\nr^2_{\0} &\leq  \left\lbrace \nu+\frac{2^4\kappa_1^8}{\nu^3}K^4 +\,\car \left[\frac{(2\tilde\kappa_0\,|b|)^\alpha}{(a\,\kappa_0)^\beta}\right]^{\frac1{\alpha-\beta}} \right\rbrace\,\nr w(t)\nr^2_{\0}.\end{align*}
Using Gronwall's inequality, we get
\[ \nr w (t)\nr^2_{\0} \leq \nr w(t=0)\nr^2_{\0}\,e^{ \left\lbrace \nu+\frac{2^4\kappa_1^8}{\nu^3}K^4 +\,\car \left[\frac{(2\tilde\kappa_0\,|b|)^\alpha}{(a\,\kappa_0)^\beta}\right]^{\frac1{\alpha-\beta}} \right\rbrace\,t},\]
which obviously leads to the desired result.
\begin{remark}\label{rem1}
In order to show the continuous dependence on the initial data and the uniqueness, we used the uniform bound given by Proposition \ref{hereisthebound}. Nevertheless, this fact can be shown without recourse to this bound by using the local integrability (in time) of $\nr \nabla u(t)\nr_2^2$ given by \eqref{mass-bound-neg} and the estimate \eqref{boundtoimprove}. The proof we presented using the uniform bound on  $\nr \nabla u(t)\nr_2^2 $ is simpler.
 \end{remark}
\begin{remark}
Observe that if $u_0,v_0 \in {\bf V} \cap L^{2\alpha+2}(\Omega)$, then $u,v \in L^2_{\rm loc}(\mathbb R^+; L^2(\Omega))$ so that $\partial_t w\in L^2_{\rm loc}(\mathbb R^+; L^2(\Omega))$ and the multiplication of the equation of $\partial_t w$ by $w$ makes sense rigorously. However, if the initial data $u_0$ and $v_0$ are only in ${\bf V}$, then $\partial_t w\in L^2_{\rm loc}(\mathbb R^+; H^{-1}(\Omega))$ and we have to consider the action of $\partial_t w$ on $w \in L^2_{\rm loc}(\mathbb R^+;H^1(\Omega))$.
\end{remark}
\section{Proof of Theorem \ref{blowupsol} }\label{blow}
\noindent In this section, we prove Theorem \ref{blowupsol} by seeking for a special solution of system \eqref{oursystem},\eqref{bc} of the form $u(x,y,z,t)=(\phi(z,t),0,0)$, with $\phi(0,t)=\phi(L,t)=0$, and $p\equiv 0$. Plugging these requirements in equation \eqref{oursystem}, we find that  $\phi(z,t)$ must satisfy
\begin{equation}\label{eqz}
\left\lbrace
\begin{array}{ll}
&\partial_t\,\phi-\nu\partial_z^2\,\phi + a|\phi|^{2\alpha}\phi+b|\phi|^{2\beta}\phi=0,\\
&\\
&\phi(0,t)=\phi(L,t)=0.
\end{array}
\right.
\end{equation}
For any smooth initial data, equation \eqref{eqz} has short time existence and uniqueness of smooth enough solutions. In the sequel, we will show that for  a special set of initial data $\phi(z,0)=\phi_0(z)$, the solution of equation \eqref{eqz} blows up in finite time. For this purpose, we need the following
\vskip6pt
\begin{lem}\label{lem}
Let $a,b\in\mathbb R$, $\phi_0(z)\geq 0$, be smooth enough initial data, and $\phi(z,t)$ be the corresponding solution of \eqref{eqz}. Let $[0,T)$ be the maximal interval of existence for $\phi(z,t)$. Then, $\phi(z,t)\geq 0$ for all $t\in[0,T)$.
\end{lem}
\vskip6pt
\begin{proof}
Let us denote by $\phi^-=\max\{0,-\phi\}$. We will show that $\phi^-\equiv 0$ for all $t\in [0,T)$. We multiply \eqref{eqz} by $\phi^-$ and integrate over $(0,L)$ and obtain
\begin{align*}
\frac12\frac{d}{dt}\,\int_0^L\,|\phi^-|^2\,dz +\nu\int_0^L\,\left|\partial_z \phi^-\right|^2\,dz +a\,\int_0^L\,|\phi^-|^{{2\alpha+2}}\,dz+b\,\int_0^L\,|\phi^-|^{{2\beta+2}}\,dz=0.
\end{align*}
This equality implies that for all $t\in[0,T)$
\[
\frac{d}{dt}\,\nr \phi^-\nr^2_{L^2(0,L)}\leq 2\left(|a|\,\nr \phi^-\nr^{2\alpha}_{L^\infty(0,L)}+|b|\,\nr \phi^-\nr^{2\beta}_{L^\infty(0,L)}\right)\nr \phi^-\nr^2_{L^2(0,L)}.
\]
Therefore, by Gronwall's inequality, we have
\[
\nr \phi^-(t)\nr_{L^2(0,L)} \leq e^{2 \int_0^t\left(|a|\,\nr \phi^-(\tau)\nr^{2\alpha}_{L^\infty(0,L)}+|b|\,\nr \phi^-(\tau)\nr^{2\beta}_{L^\infty(0,L)}\right)\,d\tau}\nr \phi^-(0)\nr^2_{L^2(0,L)}.
\]
Since $\nr \phi^-(0)\nr_{L^2(0,L)}=0$, and the exponential term is finite, we infer that $\nr \phi^-(t)\nr_{L^2(0,L)}=0$ for all $t\in[0,T)$. Consequently, $\phi(z,t)\geq 0$ for all $t\in[0,T)$.\hskip40pt
\end{proof}
\vskip6pt
\noindent Now, let $\phi_0(z)\geq 0$ be smooth enough to be chosen later. Let $T>0$ be the maximal time of existence of the solution $\phi$ of equation \eqref{eqz}, corresponding to the initial data $\phi_0$. Thanks to Lemma \ref{lem}, $\phi(z,t)\geq 0$, for all $t\in[0,T)$. We aim to show that $T<+\infty$. We proceed by contradiction and assume that $T=+\infty$. Then, we multiply equation \eqref{eqz} by $\sin \frac{\pi\,z}{L}\geq 0$, and integrate over $(0,L)$, to obtain, after integration by parts,
\begin{align}
&\frac{d}{dt}\,\int_0^L\,\phi(z,t)\,\sin \frac{\pi\,z}{L}\,dz +\nu\left(\frac\pi L\right)^2\,\int_0^L\,\phi(z,t)\,\sin \frac{\pi\,z}{L}\,dz \nonumber\\
&\hskip80pt+ a\,\int_0^L\,|\phi(z,t)|^{2\alpha}\,\sin \frac{\pi\,z}{L}\,dz + b\,\int_0^L\,|\phi(z,t)|^{2\beta}\,\sin \frac{\pi\,z}{L}\,dz=0.\label{trigestimate}
\end{align}
Now, we need the following basic fact
\vskip6pt
\begin{lem}\label{lemtrig}
Let $\psi\in L^1(0,L)$ satisfying $\psi(z)\geq 0$, a.e. on $[0,L]$; and let $\gamma\geq 1$. Then, there exists a constant $c_\gamma$ such that
\[\left(\int_0^L\,\psi(z) \,\sin \frac{\pi\,z}{L}\,dz\right)^\gamma \leq c_\gamma\,\int_0^L\,\psi(z)^{\gamma}\,\sin \frac{\pi\,z}{L}\,dz.\]
\end{lem}
\begin{proof}
The proof is a straightforward calculation. Indeed, we have
\begin{align*}
&\int_0^L\,\psi(z) \,\sin \frac{\pi\,z}{L}\,dz=\int_0^L\,\psi(z) \, \left(\sin\frac{\pi\,z}{L}\right)^\frac1\gamma\,\left(\sin\frac{\pi\,z}{L}\right)^{1-\frac1\gamma}\,dz\\
&\hskip100pt\leq \left(\int_0^L\,\psi(z)^\gamma \,\sin \frac{\pi\,z}{L}\,dz\right)^\frac1\gamma \left(\int_0^L\,\sin \frac{\pi\,z}{L}\,dz\right)^{1-\frac{1}{\gamma}}\\
&\hskip100pt=\left(\frac{2L}{\pi}\right)^{1-\frac1\gamma}\,\left(\int_0^L\,\psi(z)^\gamma \,\sin \frac{\pi\,z}{L}\,dz\right)^\frac1\gamma.
\end{align*}
The proof is achieved by setting $c_\gamma=\left(\frac{2L}{\pi}\right)^{\gamma-1}$.\hskip165pt
\end{proof}
\vskip6pt
\noindent Next, we introduce the notation $m(t):=\int_0^L\,\phi(z,t)\,\sin \frac{\pi\,z}{L}\,dz \geq 0$. Since $\phi(z,t)\geq 0$, by applying Lemma \ref{lemtrig} to equation \eqref{trigestimate} and using the fact that $a<0$, we obtain
\[\frac{d}{dt}\,m(t)\geq -\nu\left(\frac\pi L\right)^{2}\,m(t) +\frac{|a|}{c_{2\alpha+1}}\,m(t)^{2\alpha+1}-\frac{|b|}{c_{2\beta+1}}\,m(t)^{2\beta+1}.\]
Since $\alpha>\beta\geq 0$,  there exists $m^\star>0$, large enough, such that for all $m\geq m^\star$ we have
\[-\nu\left(\frac\pi L\right)^{2}\,m(t)+\frac{|a|}{c_{2\alpha+1}}\,m(t)^{2\alpha+1}-\frac{|b|}{c_{2\beta+1}}\,m(t)^{2\beta+1}>\frac{|a|}{4\,c_{2\alpha+1}}\,m(t)^{2\alpha+1}.\]
Now, we choose $m(0)\geq 2m^\star$. Therefore, by the continuity of $m(t)$, one has for short positive time that $m(t)>m^\star$. From the above, we have  for short time
\begin{equation}\frac{d}{dt}\,m(t)\geq\frac{|a|}{4\,c_{2\alpha+1}}\,m(t)^{2\alpha+1}.\label{blow-eq-inter}\end{equation}
Thus, we are able to conclude that $m(t)>m^\star$ for all $t\geq0$ and as a result, \eqref{blow-eq-inter} also holds for all $t\geq0$. Since $\alpha>0$, then by integrating \eqref{blow-eq-inter} we obtain $m(t)\geq {m(0)}\,{(1-\alpha\frac{|a|}{2\,c_{2\alpha+1}}m(0)^{2\alpha}\,t )^{-\frac{1}{2\alpha}}}$.
Hence, there exists a finite time $T^\star$ such that $\limsup_{t\rightarrow T^\star_-}\,m(t)=+\infty$. This shows the finite time blow-up and concludes the proof of Theorem \ref{blowupsol}.

\section{Proof of Theorems \ref{thmu}}\label{Proofstrongmu}\vskip6pt \noindent 
The proof follows the same line of the proof of Theorems \ref{thmweak} and \ref{thm} with the extra difficulty that consists in dealing with the difference $\mathcal I^1_h(u)-\mathcal I^1_h(v)$. From this point onward, we will no longer split the proofs depending on the sign of $b$ and use the characteristic function $\mathds 1_{\{b<0\}}$ for shortness.

\subsection{Local-in-time solutions} \label{localintimesolution}
\vskip6pt
\noindent Let $u$ be  a strong solution of system \eqref{oursystem},\eqref{pbc}, our goal is to prove the existence and uniqueness of strong solutions to the following system
\begin{equation*}\mathcal S_1:\quad
\left\lbrace\begin{array}{ll}
&\partial_t\,v-\nu\,\Delta\,v +(v\cdot\nabla)\,v +\nabla\,q+a\,|v|^{2\alpha}\,v+b\,|v|^{2\beta}\,v=\varphi -\mu\,\mathcal I^1_h(v),\\ \\
& \nabla\cdot v=0,\;v\vert_{t=0}=v_0, \\\\
&v(x+L,y,z,t)=v(x,y+L,z,t)=v(x,y,z+L,t) = v(x,y,z,t), \\\\
&q(x+L,y,z,t)=q(x,y+L,z,t)=p(x,y,z+L,t) =q(x,y,z,t),
\end{array}
\right.
\end{equation*}
where we set $\varphi :=f+\mu\,\mathcal I_{h}(u)$. Before going further, let us recall that  $u\in  C^0_{\rm b}(\mathbb R^+;{\bf V})$, by Theorem \ref{thm}, so that thanks to the triangular inequality and \eqref{b1}, we have
\begin{align}\label{hastoputanumberhere}
\nr \mathcal I^1_h(u)\nr_{{\0}} &\leq \nr u- \mathcal I^1_h(u)\nr_{{\0}}+\nr u\nr_{{\0}}\leq \sqrt {c_0}\,h\,\nr \nabla u\nr_2 + \nr u\nr_2.
\end{align}
In particular, this shows that $\mathcal I^1_h(u) \in C^0_{\rm b}(\mathbb R^+;{\bf H})$, which in turn implies that $\varphi \in  C^0_{\rm b}(\mathbb R^+;{\bf H})$, and that for all $t\geq 0$ it holds thanks to \eqref{mass-bound-neg} and Proposition \ref{hereisthebound}
\begin{equation}\label{boundvarphi}
\nr \varphi\nr_{\0} \leq \sqrt M := \nr f\nr_{\infty,2}+ \mu \,\sqrt{c_0}\,h\,K + \mu\,\left(\nr u_0\nr^2_{\0}+\frac{2}{\nu}\,\left(\nr f\nr^2_{\infty,2} + \eta_1\right)\right)^\frac12.
\end{equation}
\vskip6pt
\noindent As in section \ref{existence galerkin}, we start by showing the existence of local (in time) solution to the system $\mathcal S_1$ by the mean of Faedo-Galerkin approximation procedure. More precisely, as in \eqref{eqpart0}, we expand $v$ as follows
\begin{equation}
{\bf v}_m(x,t)=\sum_{k=1}^m g_m^i(t)\,\mathcal W_i(x) \in \mathbf V_m,\label{eqpart0mu}
\end{equation}
$g_m^i= \int_\Omega \mathbf v_m\cdot \mathcal W_i\,dx$, for $i=1,2, \ldots, m$, are the unknown coefficients, solving the following system of ordinary differential equations in $\mathbf V_m$
\begin{align}
&\frac d{dt}\,\int_\Omega \mathbf v_m\cdot \mathcal W_k\,dx+\nu\,\int_\Omega \nabla\mathbf v_m\cdot \nabla\mathcal W_k\,dx + \int_\Omega (\mathbf v_m\cdot\nabla)\mathbf v_m\cdot \mathcal W_k\,dx \nonumber \\&+ a\,\int_\Omega |\mathbf v_m|^{2\alpha} \mathbf v_m\cdot\mathcal W_k\,\,dx +b\int_\Omega |\mathbf v_m|^{2\beta} \mathbf v_m\cdot\mathcal W_k\,dx\nonumber \\ &\hskip30pt= \int_\Omega \varphi\cdot\mathcal W_k\,dx -  \mu\,\int_\Omega  \mathcal I^1_{h}(\mathbf v_m)\cdot\mathcal W_k\,dx,\label{eqpart1mu}\\
\label{eqpart2mu}&g_m^k(t=0)=\int {\mathbf v_0}\cdot \mathcal W_k\,dx,
\end{align}
for all $k=1,\ldots,m$. The vector field in system \eqref{eqpart0mu},\eqref{eqpart2mu} being obviously locally Lipschitz, the system admits a unique solution $g_m^k(t)\in C^1([0,T_m])$, for some small interval $[0,T_m]\subset [0,T]$. As in the previous section, we will perform formal estimates using system $\mathcal S_1$ instead of the Fadeo-Galerkin system  for simplicity.
\subsection{Global-in-time solutions} \label{globalintimesolution}
\vskip6pt
\noindent
We multiply the first equation of system $\mathcal S_1$ by $v$ and integrate over $\Omega$. Using Cauchy-Schwarz and Young inequalities we obtain
\begin{align*}
\nonumber \frac12\,\frac{d}{dt}\,\nr v\nr_2^2 &+\nu\,\nr \nabla v\nr_2^2 +a\,\nr v\nr_{{{2\alpha+2}}}^{{2\alpha+2}} +b\,\nr v\nr_{{{2\beta+2}}}^{{2\beta+2}}\\ \nonumber
& \leq \frac{1}{\mu}\,\nr \varphi\nr_2^2-\frac34 \,\mu\,\nr v\nr_2^2 + \mu\,\int_{\Omega}\,\left(v(t,x)-\mathcal I^1_h(v(t,x))\right)\cdot v(t,x)\,dx\\
&\leq \frac{1}{\mu}\,\nr \varphi\nr_2^2-\frac12\, \mu\,\nr v\nr_2^2 + {\mu\,c_0\,h^2}\,\nr \nabla v\nr_2^2.
\end{align*}
Now, thanks to \eqref{relationalphabeta}, \eqref{boundvarphi}, and the fact that $2{\mu\,c_0\,h^2}\leq\nu$, we obtain
\begin{align}
\frac{d}{dt}\,\nr v\nr_2^2 +\mu\,\nr u\nr_2^2+\nu\,\nr \nabla v\nr_2^2 &+(2-\car)\,a\,\nr v\nr_{{{2\alpha+2}}}^{{2\alpha+2}} \leq 2\left(\frac{M}{\mu}+ \eta_2\right),\label{thisalsotoingrate}
\end{align}
where we set $\eta_2:=\car\,|b|\,\left[\frac{2|b|}{a}\right]^{{\frac{\beta+1}{\alpha-\beta}}}\,|\Omega|$. In particular, we have
\begin{align*}
\frac{d}{dt}\,\nr v\nr_2^2 +\mu\,\nr v\nr_2^2  \leq 2\left(\frac{M}{\mu}+ \eta_2\right),
\end{align*}
and by the virtue of Gronwall's inequality we get
\begin{align} \label{massboundv+}
\nr v(t)\nr^2_\0 &\leq  \nr v_0\nr_\0^2\,e^{-\mu\,t} +  \frac2{\mu}\,\left(\frac{M}{\mu}+ \eta_2\right).
\end{align}
Notice that integrating  \eqref{thisalsotoingrate} with respect to time, we obtain
\begin{align}
 \nu\,\st\nr\nabla\,v(s)\nr^2_{\0}\,ds &+(2-\car)\,a\,\st\,\nr v(s)\nr^{{2\alpha+2}}_{{{2\alpha+2}}}\,ds \nonumber\\
&\hskip60pt\leq \nr v_0\nr^2_{\0} + 2\left(\frac{M}{\mu}+ \eta_2\right)\,t.\label{tousefornegbound}
\end{align}
Now, we  find a bound on $\nr \nabla v\nr_\0$. For this purpose, we multiply the first equation of system $\mathcal S_1$ by $-\Delta \,v$ and integrate over $\Omega$ (see Remark \ref{Remreg}). We obtain
\begin{align*}
\nonumber&\frac12\frac{d}{dt}\,\nr \nabla v\nr_2^2 +\nu\,\nr \Delta v\nr_2^2+\s\,(v\cdot\nabla)\,v\,\cdot\,(-\Delta\,v)\,dx +a\,(1+2\alpha)\,\nr |v|^\alpha\,\nabla v\nr_2^2 \\ &+b\,(1+2\beta)\,\nr |v|^\beta\,\nabla v\nr_2^2 =\s\,\varphi\,\cdot\,(-\Delta\,v)\,dx- \mu\,\s\,\mathcal I^1_h(v)\cdot\,(-\Delta\,v)\,dx.
\end{align*}
Now, using Cauchy-Schwarz inequality and \eqref{b1}, we can write
\begin{align}\label{anologue0}
\nonumber- \mu\,\s\,\mathcal I^1_h(v)\cdot\,(-\Delta\,v)\,dx&=\mu\,\s\,\left(v-\mathcal I^1_h(v)\right)\cdot\,(-\Delta\,v)\,dx- \mu\,\s\,v\cdot\,(-\Delta\,v)\,dx\nonumber\\
&\leq  \left(\frac{\mu^2\,c_0\,h^2}{4\epsilon}-\mu\right)\,\nr \nabla\,v\nr^2_2 + \epsilon\,\nr\Delta v\nr_2^2 .
\end{align}
Using this, and the fact that $2\mu\,c_0\,h^2\leq \nu$, and proceeding as in the previous section, by using \eqref{cumbersome} and \eqref{relationbad}, with different optimization in the $\epsilon'$s, we get
\begin{align}
&\frac{d}{dt}\,\nr \nabla v\nr_2^2 +\nu\,\nr \Delta v\nr_2^2+a\,(1+2\alpha)\,\nr |v|^\alpha\,\nabla v\nr_2^2 \leq \frac2\nu\,\nr \varphi\nr_2^2 + \mathcal A_2\,\nr \nabla v\nr_2^2,\label{tointegforneg}
\end{align}
where
\begin{align}
\mathcal A_2:= -\frac\mu2+ \left\lbrace \left[\frac{\nu^\alpha\,a\,(1+2\alpha)}{3^\alpha\,(1+\car)}\right]^{\frac1{1-\alpha}} + \car\left[\frac{4^\beta[|b|(1+2\beta)]^\alpha}{[a(1+2\alpha)]^\beta}\right]^{\frac1{\alpha-\beta}} \right\rbrace.\label{defnA2}
\end{align}
In particular, we have
\begin{align*}
&\frac{d}{dt}\,\nr \nabla v\nr_2^2  \leq \frac2\nu\,\nr \varphi\nr_2^2 + \mathcal A_2\,\nr \nabla v\nr_2^2.
\end{align*}
Integrating this inequality with respect to time, and using \eqref{tousefornegbound} we obtain
\begin{align}\label{i2touseforunigrad}
\nr \nabla v(t)\nr_2^2  \leq \frac{\mathcal A_2}{\nu}\,\nr v_0\nr^2_{\0}+\nr \nabla v_0\nr_2^2+\frac{2}\nu\,\left(M\left(1+ \frac{\mathcal A_2}{\mu} \right) + \mathcal A_2\,\eta_2\right)\,t  .
\end{align}
Moreover, integrating \eqref{tointegforneg} and using \eqref{tousefornegbound}, we get
\begin{align*}
\nu \, \int_0^{t}\,\nr \Delta v(s)\nr_2^2\,ds&+{a(1+2\alpha)}\,\int_0^{t} \nr |v(s)|^\alpha\,\nabla v(s)\nr_2^2 \nonumber\\ &\leq \frac{\mathcal A_2}{\nu}\,\nr v_0\nr^2_{\0}+\nr \nabla v_0\nr_2^2+\frac{2}\nu\,\left(M\left(1+ \frac{\mathcal A_2}{\mu} \right) + \mathcal A_2\,\eta_2\right)\,t.
\end{align*}
The global existence of solutions to $\mathcal S_1$ follows as for Theorem \ref{thm}. Let us mention that the calculation we performed in this proof is formal. Nevertheless, it can be done rigorously using the Faedo-Galerkin system, especially when multiplying the first equation of system $\mathcal S_1$ by $-\Delta \,v$ and integrating over $\Omega$ and we refer the reader to Remark \ref{Remreg}.
\vskip6pt
\noindent Before going further, we show the following
\begin{prop}\label{propforvgradbound}
Let $f\in L^\infty(\mathbb R^+;{\bf H})$, $u$ the strong solution of system \eqref{oursystem},\eqref{pbc} given by Theorem \ref{thm}, $v_0\in {\bf V}, \alpha>1, 0\leq \beta<\alpha , a>0$ and $b\in\mathbb{R}$. If $2\mu\,c_0\,h^2\leq \mu$, then, there exists $P>0$ such that the solutions of (\ref{oursystemu}-\ref{pbcmu}) satisfy
\begin{equation*}
\nr \nabla v(t)\nr_2 \leq P:= \left\lbrace
\begin{array}{lcl}
P_1& {\rm if} & 0\leq t\leq 1,\\
P_2& {\rm if} & t>1 ,
\end{array}
\right.
\end{equation*}
where
\begin{align*}
P_1&:=\left[ \frac{\mathcal A_2}{\nu}\,\nr v_0\nr^2_{\0}+\nr \nabla v_0\nr_2^2+\frac{2}\nu\,\left(M\left(1+ \frac{\mathcal A_2}{\mu} \right) + \mathcal A_2\,\eta_2\right) \right]^\frac12, \\
P_2&:= \left[  \frac1\nu \left(\mathcal A_2+1\right)\,\nr v_0\nr^2_{\0} +\frac2\nu\left\lbrace M + \left(\mathcal A_2+1\right)\,\left(1+\frac1\mu\right) \,\left(\frac M\mu+\eta_2\right)\right\rbrace\right]^\frac12.
\end{align*}
and $\mathcal A_2$ is defined in \eqref{defnA2} and $\eta_2:=\car\,2|b|\,\left[\frac{2|b|}{a}\right]^{{\frac{\beta+1}{\alpha-\beta}}}\,|\Omega|$.
\end{prop}
\begin{proof}
The proof is in the same spirit as the one of Proposition \ref{hereisthebound}. On the one side, if $t\in\left[0,1\right]$, then we use \eqref{i2touseforunigrad} to get
\begin{align*}
\nr \nabla v(t)\nr_2^2  \leq \frac{\mathcal A_2}{\nu}\,\nr v_0\nr^2_{\0}+\nr \nabla v_0\nr_2^2+\frac{2}\nu\,\left(M\left(1+ \frac{\mathcal A_2}{\mu} \right) + \mathcal A_2\,\eta_2\right).
\end{align*}
On the opposite side, let $t> 1$ and introduce $s\in \left[t-1,t\right]$. Using \eqref{tointegforneg}, we have
\begin{align}\label{sss}
\nr \nabla\,v(t)\nr^2_{\0} &\leq  \nr \nabla\,v(s)\nr^2_{\0} +\frac{2\,M}{\nu} +\mathcal A_2\,\int_{t-1}^t\nr\nabla\,v(\tau)\nr^2_{\0}\,d\tau.
\end{align}
Now, we integrate \eqref{thisalsotoingrate} over $ \left[t-1,t\right]$ and use \eqref{massboundv+} to get
\begin{align*}
\nu\, \int_{t-1}^t\nr\nabla\,v(s)\nr^2_{\0}\,ds &\leq 2\left(1+\frac1\mu\right)\,\left(\frac{M}{\mu}+ \eta_2\right)+e^{-\mu\left(t-1\right)}\, \nr v_0\nr^2_{\0}.
\end{align*}
Therefore, \eqref{sss} becomes now
\begin{align*}
\nr \nabla\,v(t)\nr^2_{\0} \leq\nr \nabla\,v(s)\nr^2_{\0}&+ \frac{\mathcal A_2}{\nu}\,e^{-\mu\left(t-1\right)}\,\nr v_0\nr^2_{\0} \\
 &+ \frac2\nu\left\lbrace M + \mathcal A_2\,\left(1+\frac1\mu\right) \,\left(\frac M\mu+\eta_2\right)\right\rbrace.
\end{align*}
Next, we integrate this inequality with respect to $s$ to obtain
\begin{align*}
\nr \nabla\,v(t)\nr^2_{\0} &\leq   \int_{t-1}^t\nr \nabla\,v(s)\nr^2_{\0}\,ds + \frac{\mathcal A_2}{\nu}\,e^{-\mu\left(t-1\right)}\,\nr v_0\nr^2_{\0} \\
 &+ \frac2\nu\left\lbrace M + \mathcal A_2\,\left(1+\frac1\mu\right) \,\left(\frac M\mu+\eta_2\right)\right\rbrace \\
&\leq  \frac1\nu \left(\mathcal A_2+1\right)\,e^{-\mu\left(t-1\right)}\,\nr v_0\nr^2_{\0} +\frac2\nu\left\lbrace M + \left(\mathcal A_2+1\right)\,\left(1+\frac1\mu\right) \,\left(\frac M\mu+\eta_2\right)\right\rbrace.
\end{align*}
This leads to the desired result and achieves the proof of Proposition \ref{propforvgradbound}.
\end{proof}
\noindent Now, we show  the fact that $v\in L^\infty_{\rm loc}(\mathbb R^+;L^{2\alpha+2}(\Omega))$ and $\partial_t v\in L^2_{\rm loc}(\mathbb R^+, L^2(\Omega))$ if $v_0 \in {\bf V} \cap L^{2\alpha+2}(\Omega)$. Actually, it follows the same argument as in the proof of Theorem \ref{thm} in section \ref{sectionthm12} and the extra term to deal with is the following
\begin{align*}
\s \mathcal I_h^1(v)\,\partial_t v \,dx &\leq \epsilon \, \nr \partial_t v\nr_2^2 +\frac1{4\epsilon} \,\nr\mathcal I_h^1(v)\nr_2^2.
\end{align*}
Now, optimizing in $\epsilon$, the term $\epsilon \, \nr \partial_t v\nr_2^2$ can be absorbed in the left hand side of the equivalent expression of \eqref{themainfortimederivative} for $v$. Eventually, inequality \eqref{hastoputanumberhere} shows clearly that $\nr\mathcal I_h^1(v)\nr_2^2$ is integrable in time which leads to the desired result.
\subsection{Continuous dependence on the initial data and uniqueness of solutions}
\vskip6pt
\noindent
Equivalently to (\ref{oursystem}--\ref{pbc}), before we show the continuous dependence on the initial data and the uniqueness of such solutions, let us prove the following
\noindent Now, we turn to the proof of the continuous dependence of the solutions $v(t)$, obtained in the previous section, on the initial data and their uniqueness. We mention that Remark \ref{rem1} is still valid in the case of system $\mathcal S_1$. Let $u$ be  a strong solution of system \eqref{oursystem},\eqref{pbc} and $v_1$ and $v_2$ two solutions of (\ref{oursystemu}--\ref{pbcmu}). Furthermore, let $w=v_1-v_2$, then $w$ satisfies
\begin{align}
\partial_t\,w-\nu\,\Delta\,w +(w\cdot\nabla)\,v_1 +(v_2\cdot\nabla)\,w +\nabla\,(p_{v_1}-p_{v_2})+a\,\left(|v_1|^{2\alpha}v_1-|v_2|^{2\alpha}v_2\right)\nonumber \\ +b\,\left(|v_1|^{2\beta}v_1-|v_2|^{2\beta}v_2\right)=-\mu\,\mathcal I^1_h(w).\label{diffequation}
\end{align}
\noindent Now, clearly $w\in L^2(\mathbb R^+;H^1(\Omega))$, and since $\partial_t v_1,\partial_tv_2 \in L^2_{\rm loc}(\mathbb R^+;H^{-1}(\Omega))$, then we have also $\partial_t w \in L^2_{\rm loc}(\mathbb R^+;H^{-1}(\Omega))$. Therefore, the action of the equation on $w$ leads to
\begin{align}
\frac12\frac d{dt}\,\nr w (t)\nr^2_{\0} +\s [(w\cdot\nabla)\,v_1]\cdot w \,dx+ \nu\,\nr \nabla\,w\nr^2_{\0} +a\,\s\left(|v_1|^{2\alpha}v_1-|v_2|^{2\alpha}v_2\right)\cdot w\,dx\nonumber \\ +b\,\s\left(|v_1|^{2\beta}v_1-|v_2|^{2\beta}v_2\right)\cdot w\,dx=-\mu\,\int_\Omega\, \mathcal I^1_h(w)\cdot w\,dx. \label{diffequationw}
\end{align}
We will need
\begin{align*}
- \mu\,\s\,\mathcal I^1_h(w)\cdot\,w\,dx &\leq \frac{\mu\,c_0\,h^2}{4\epsilon} \,\nr \nabla\,w\nr^2_{\0} +(\mu\epsilon-\mu)\,\nr w\nr^2_\0,
\end{align*}
Now, gathering this inequality and (\ref{positivity}--\ref{diff}), and optimizing in the $\epsilon'$s,  we obtain
\begin{align}
\frac d{dt}\,\nr w (t)\nr^2_{\0} &+\nu\,\nr \nabla w\nr_2^2+ (2-\car)\,{a\,\kappa_0}\,\nr \left(|u|+|v|\right)^{\alpha} |w|\nr^2_{\0}\nonumber\\
\label{helpconvergenceint1l2}& \leq \left\lbrace -\mu+\frac12\,\nu+\frac{2^7\kappa_1^8}{\nu^3}P^4 +\,\car \left[\frac{(2\tilde\kappa_0\,|b|)^\alpha}{(a\,\kappa_0)^\beta}\right]^{\frac1{\alpha-\beta}} \right\rbrace\,\nr w(t)\nr^2_{\0},
\end{align}
Gronwall's inequality leads to the desired result. In particular, we have
\begin{align}
\nonumber \nu \,\int_t^{t+T}\nr \Delta\,v(s)\nr^2_{\0}\,ds &+ {a\,(1+2\alpha)}\,\int_t^{t+T}\nr |v(s)|^\alpha\,\nabla\,v(s)\nr^2_{L^2}\,ds \\ &\leq \frac2{\nu}\,\nr \varphi\nr^2_{\infty,2} + (\mathcal A_2 +1) \,{P}^2.\label{mainonep}
\end{align}
Observe that if $v_0\in {\bf V} \cap L^{2\alpha+2}(\Omega)$, then $\partial_t v_1,\partial_t v_2 \in L^2_{\rm loc}(\mathbb R^+;L^2(\Omega))$ so that $\partial_t w \in L^2_{\rm loc}(\mathbb R^+;L^2(\Omega))$ and the inner $L^2$ inner product of $\partial_t w$ with $w$ is well defined and no regularization is needed.

\section{Proof of Theorem \ref{thmconI1}}
\subsection{Convergence in $L^2$ norm}\label{leconvsubsec}
\vskip6pt
\noindent
In this section, we show the convergence of the difference $u(t)-v(t)$  to zero as $t\to +\infty$ in the $L^2$ norm where $u(t)$ is a strong solution of system \eqref{oursystem},\eqref{pbc} ensured by Theorems \ref{thmweak} and \ref{thm} and $v(t)$ the solution of system $\mathcal S_1$ with an interpolant $\mathcal I^1_h$. This will be achieved by suitable assumptions and the smallness of $h$ and that $\mu$ is large enough. Let $w=u-v$, then subtracting the first equation of system $\mathcal S_1$ from the first equation of system \eqref{oursystem},\eqref{pbc} leads to equation \eqref{diffequation} (with $v_1$ replaced by $u$ and $v_2$ replaced by $v$). We multiply this equation by $w$ and integrate over $\Omega$ to get
\begin{align*}
&\frac12\frac d{dt}\,\nr w (t)\nr^2_{\0}+\nu\,\nr \nabla\,w\nr^2_{\0} +\s [(w\cdot\nabla)\,u]\cdot w \,dx +\,a\,\s\left(|u|^{2\alpha}u-|v|^{2\alpha}v\right)\cdot w\,dx\nonumber \\ &+b\,\s\left(|u|^{2\beta}u-|v|^{2\beta}v\right)\cdot w\,dx=-\mu\,\int_\Omega\, \mathcal I^1_h(w)\cdot w\,dx. \label{diffequationwz}
\end{align*}
Therefore, \eqref{helpconvergenceint1l2} still holds true with $P$ replaced by $K$. Observe that $K$ is independent of $\mu$. Therefore if we choose $\mu$ large enough such that
\[\mu>2\left\lbrace \frac12\,\nu+\frac{2^7\kappa_1^8}{\nu^3}K^4 +\,\car \left[\frac{(2\tilde\kappa_0\,|b|)^\alpha}{(a\,\kappa_0)^\beta}\right]^{\frac1{\alpha-\beta}} \right\rbrace,\]
then, using Gronwall's inequality we obtain
\[\nr w(t)_2^2 \leq \nr w(t=0)\nr_2^2\,e^{-\frac\mu2 \,t}.\]
This shows the convergence of $u(t)-v(t)$ to zero, as $t\to +\infty$, in the $L^2$ norm.
\subsection{Convergence in $H^1$ norm}\label{convergenceinH1norminterpolant1}
\vskip6pt
\noindent
In this section, we show the last assertion of Theorem \ref{thmu}, namely the convergence of the difference $u(t)-v(t)$ to zero, as $t\to +\infty$, in the $H^1$ norm where $u(t)$ is a strong solution of system \eqref{oursystem},\eqref{pbc} ensured by Theorems \ref{thmweak} and \ref{thm} and $v(t)$ the solution of system $\mathcal S_1$.
\noindent Now, since we restrict the range of $\alpha$ to $1<\alpha <2$, therefore we have $\partial_t u,\partial_t v \in L^2_{\rm loc}(\mathbb R^+;L^2(\Omega))$, therefore $\partial_t w \in L^2_{\rm loc}(\mathbb R^+;L^2(\Omega))$ and $-\Delta w \in L^2_{\rm loc}(\mathbb R^+, L^2(\Omega))$ so that the $L^2$ inner product of  \eqref{diffequation} with $-\Delta w$ leads to
\begin{align}
\nonumber&\frac12\frac d{dt}\,\nr \nabla\,w (t)\nr^2_{\0} + \nu\,\nr \Delta\,w\nr^2_{\0} +\s [(w\cdot\nabla)\,u]\cdot (-\Delta\,w) \,dx+ \s [(v\cdot\nabla)\,w]\cdot (-\Delta\,w) \,dx\\ & +a\,\s\left(|u|^{2\alpha}u-|v|^{2\alpha}v\right)\cdot (-\Delta\,w)\,dx  +b\,\s\left(|u|^{2\beta}u-|v|^{2\beta}v\right)\cdot (-\Delta\,w)\,dx\nonumber\\ &=-\mu\,\int_\Omega\, \mathcal I^1_h(w)\cdot (-\Delta\,w)\,dx. \label{diffequationwzw}
\end{align}
We need the following estimate
\begin{align} \label{bad1}
&\s [(w\cdot\nabla)\,u]\cdot (-\Delta\,w) \,dx\leq \nr \nabla u\nr_2\,\nr w\nr_\infty\,\nr \Delta w\nr_2  \leq \kappa_2\,\nr \nabla u\nr_2\,\nr w\nr^\frac12_{H^1}\,\nr w\nr^\frac12_{H^2}\,\nr \Delta w\nr_2\nonumber \\
&\hskip50pt\leq \kappa_2\,\nr \nabla u\nr_2\,\nr w\nr^\frac12_{H^1}\,\nr w\nr^\frac32_{H^2} \leq \left[\frac{\kappa_2^4}{\epsilon^3}\,\nr \nabla u\nr_2^4 +\epsilon \right]\,\nr w\nr^2_{H^1} + \epsilon\,\nr \Delta w\nr^2_{2}.
\end{align}
 Also, we need
\begin{align}
& \s [(v\cdot\nabla)\,w]\cdot (-\Delta\,w) \,dx\leq \nr v\nr_{\infty}\,\nr \nabla w\nr_2\,\nr \Delta w\nr_2 \nonumber\\
&\hskip20pt\leq {\kappa_2}\,\left(\frac1{4\epsilon}\,\nr v\nr_{H^1} +\epsilon\,\nr v\nr_{H^2} \right)\,\nr \nabla w\nr_2\,\nr \Delta w\nr_2\nonumber\\
& \hskip20pt\leq \frac{\kappa^2_2}{4\epsilon_1}\left(\frac1{16\epsilon^2}+\epsilon^2\right) \nr v\nr^2_{H^1}\,\nr \nabla w\nr_2^2+2\epsilon_1 \nr \Delta w\nr_2^2+ \frac{\kappa_2^2\epsilon^2}{4\epsilon_1}\,\nr \Delta v\nr^2_{2}\,\nr \nabla w\nr_2^2.\label{thisfirst}
\end{align}
For later use (see below), it is crucial to avoid having  the term $\nr \Delta v\nr^2_{2}$ in the right hand side  of this inequality. For this purpose, we use  the triangular inequality to get
\begin{align}\label{transfer}
\nr \Delta v\nr^2_{2}\,\nr \nabla w\nr_2^2 &\leq 2\left(\nr \Delta w\nr^2_{2}+\nr \Delta u\nr^2_{2}\right)\,\nr \nabla w\nr_2^2\nonumber\\
&\leq 2\nr \Delta u\nr^2_{2}\,\nr \nabla w\nr_2^2 + 4\left(\nr \nabla u\nr^2_{2}+\nr \nabla v\nr^2_{2}\right)\,\nr \Delta w\nr_2^2.
\end{align}
Therefore, inequality \eqref{thisfirst} can be replaced by
\begin{align}\label{bad2}
&\s [(v\cdot\nabla)\,w]\cdot (-\Delta\,w) \,dx\leq   \left(2\epsilon_1+ \frac{\kappa_2^2\epsilon^2}{\epsilon_1}\,\left[\nr \nabla u\nr^2_{2}+\nr \nabla v\nr^2_{2}\right]\right)\,\nr \Delta w\nr_2^2  \nonumber\\
&\hskip50pt+\frac{\kappa^2_2}{4\epsilon_1}\left(\frac1{16\epsilon^2}+\epsilon^2\right) \nr v\nr^2_{H^1}\,\nr \nabla w\nr_2^2+ \frac{\kappa_2^2\epsilon^2}{2\,\epsilon_1}\,\nr \Delta u\nr^2_{2}\,\nr \nabla w\nr_2^2.
\end{align}
The difficulty in showing the convergence in the $H^1$ norm, compared to the $L^2(\Omega)$ one, consists in dealing with the power terms. Indeed, on the one side, there exists $\kappa_4(\alpha)$ such that
\begin{align*}
\s\left(|u|^{2\alpha}u-|v|^{2\alpha}v\right)\cdot (-\Delta w)\,dx
&\leq\kappa_4(\alpha)\,\s\left(|u|^{2\alpha}+|v|^{2\alpha}\right)\,|w||-\Delta w|\,dx,
\end{align*}
For simplicity, we set $\kappa_4:=\kappa_4(\alpha) \geq \kappa_4(\beta)$. In particular, using this and Young's inequality, we obtain
\begin{align}\label{long}
&a\,\s\left(|u|^{2\alpha}u-|v|^{2\alpha}v\right)\cdot (-\Delta w)\,dx+b\,\s\left(|u|^{2\alpha}u-|v|^{2\alpha}v\right)\cdot (-\Delta w)\,dx \nonumber\\
&\hskip15pt\leq \kappa_4\,(a+\epsilon |b|)\s\left(|u|^{2\alpha}+|v|^{2\alpha}\right) |w||-\Delta w|\,dx+2\kappa_4\,|b|\,\left[\frac1\epsilon\right]^{\frac\beta{\alpha-\beta}}\,\nr \nabla w\nr_2^2.
\end{align}
In order to deal with the right hand side of this inequality, we will need to restrict the range of $\alpha$. In fact, on the one hand we have
\begin{align}\label{partonealpha}
\kappa_4\s\,|u|^{2\alpha} |w| |-\Delta w|\,dx &\leq \kappa_4\,\nr u\nr^{2\alpha}_{{4\alpha}}\,\nr w\nr_{\infty}\,\nr \Delta w\nr_\0\nonumber \\
&\leq \kappa_4\,\kappa_2\nr u\nr^{2\alpha}_{{4\alpha}}\,\nr w\nr^\frac12_{H^1}\,\nr  w\nr^\frac32_{H^2}\nonumber \\
&\leq \left(\frac{\kappa^4_2\,\kappa_4^4}{4\,\epsilon^3}\, \nr u\nr^{8\alpha}_{{4\alpha}}+\epsilon\right)\,\nr w\nr^2_{H^1} +\epsilon \nr \Delta w\nr^2_\0\nonumber \\
&\leq \left(\frac{\kappa^4_2\,\kappa_4^4}{4\,\epsilon^3}\,|\Omega|^{8\alpha}\, \nr u\nr^{8\alpha}_{6}+\epsilon\right)\,\nr w\nr^2_{H^1} +\epsilon \nr \Delta w\nr^2_\0\nonumber \\
&\leq \left(\frac{\kappa^4_2\,\kappa_3^{8\alpha}\,\kappa_4^4}{4\,\epsilon^3}\,|\Omega|^{8\alpha}\, \nr u\nr^{8\alpha}_{H^1}+\epsilon\right)\,\nr w\nr^2_{H^1} +\epsilon \nr \Delta w\nr^2_\0.
\end{align}
In order to pass from the third to the fourth line above, we used H\"older inequality, which requires $\alpha \leq \frac32$. On the other hand, let $\alpha>\frac32$, then we can interpolate $4\alpha$ between $6$ and $+\infty$, that is
\begin{align*}
&\kappa_4\,\s\,|u|^{2\alpha}|w| |-\Delta w|\,dx\leq \frac{\kappa^4_2\,\kappa_4^4}{4\,\epsilon^3}\, \nr u\nr^{8\alpha}_{{4\alpha}}\,\nr \nabla w\nr^2_{\0} +\epsilon\nr \Delta w\nr^2_\0 \nonumber \\
&\hskip40pt\leq  \frac{\kappa^4_2\,\kappa_4^4}{4\,\epsilon^3}\,\nr u\nr^{12}_{6}\,\nr u\nr^{4\,{(2\alpha-3)}}_{{\infty}}\,\nr \nabla w\nr^2_{\0} +\epsilon\nr \Delta w\nr^2_\0 \nonumber\\
&\hskip40pt\leq  \frac{\kappa^{(8(\alpha-1))}_2\,\kappa_3^{12}\,\kappa_4^4}{4\,\epsilon^3}\,\nr u\nr^{2(2\alpha+3)}_{H^1}\,\nr u\nr^{2{(2\alpha-3)}}_{H^2}\,\nr \nabla w\nr^2_{\0} +\epsilon\nr \Delta w\nr^2_\0.
\end{align*}
Now, we use the fact that
\[\nr u\nr^{2(2\alpha+3)}_{H^1}\,\nr u\nr^{2(2\alpha-3)}_{H^2} \leq \left[\frac1\epsilon\right]^{\frac{2\alpha-3}{4-2\alpha}}\,\nr u\nr^{\frac{2\alpha+3}{2-\alpha}}_{H^1} + \epsilon \nr u\nr^2_{H^2},\]
which makes sense only when $2(2\alpha-3) < 2$, that is $\alpha<2$ to get
\begin{align}\label{parttwoalpha}
&\kappa_4\,\s\,|u|^{2\alpha}|w| |-\Delta w|\,dx\leq
 \frac{\kappa^{8(\alpha-1)}_2\,\kappa_3^{12}\,\kappa_4^4}{4\,\epsilon^3}\,\left(\left[\frac1{\tilde\epsilon}\right]^{\frac{2\alpha-3}{4-2\alpha}}\,\nr u\nr^{\frac{2\alpha+3}{2-\alpha}}_{H^1} + \tilde \epsilon \nr u\nr^2_{H^1}\right)\, \nr \nabla w\nr^2_{2} \nonumber\\
&\hskip60pt+ \frac{\kappa^{8(\alpha-1)}_2\,\kappa_3^{12}\,\kappa_4^4\,\tilde \epsilon}{4\,\epsilon^3}\,\nr \Delta u\nr_2^2\,\nr \nabla w\nr_2^2 +\epsilon \nr \Delta w\nr_2^2.
\end{align}
\noindent Obviously estimates \eqref{partonealpha} and \eqref{parttwoalpha}  hold if $u$ is replaced by $v$  and $K$ replaced by $P$. However, since $P$ depends on $\mu$ (see Proposition \ref{propforvgradbound} and \eqref{boundvarphi}), we will have to solve when $1<\alpha\leq \frac32$ an inequality of type $\mu > {\rm Const} \,\mu^{8\alpha} + {\rm Const\,2} \,\mu^{2(2\alpha+3)} + ...$ which does not have a solution most of the time. At least, in our case we cannot exhibit precisely the solution's range if it exists. The hypothesis on the initial data $\nr u_0\nr_{H^1} \leq \tilde K$, $\nr v_0\nr_{H^1} \leq \tilde K$  is introduced in order to get over this difficulty. Indeed, by continuity of $\nr v(t)\nr_{H^1}$, there exists a short time interval $[0,\overline T)$ such that for all $t\in[0,\overline T)$, it holds $\nr v(t)\nr^2_{H^1} \leq 3\tilde K^2$. In the sequel, arguing by contradiction, we will show that, actually, we have $\overline T=+\infty$. Obviously, this assumption can be removed by assuming $\nr \nabla v_0\nr_2 \leq K $ since $\nr v(t)\nr_2$ is bounded. We work with the first assumption for lightness of notation.
\vskip6pt
\noindent In the sequel, we will need the following version of Gronwall's Lemma for which we refer to \cite{Titi}
\begin{lem}\label{Gronwall}
Let $\zeta(t)$ be an absolutely continuous and locally integrable function satisfying $\frac d{dt}\,\zeta(t) + \xi(t)\,\zeta(t) \leq 0$. Assume that for a fixed $s>0$
\[\liminf_{t\to\infty}\,\int_t^{t+s}\, \xi(s)\,ds \geq \delta\quad {\rm and}\quad  \limsup_{t\to\infty}\,\int_t^{t+s}\, \xi^-(s)\,ds<\infty,\]
where $\delta>0$ and $\xi^-=\max\{-\alpha,0\}$. Then, $\zeta(t) \to 0$ exponentially, as $t\to\infty$.
\end{lem}
\noindent Now, we assume that $\overline T$ is the the maximal finite time such that $\nr v(t)\nr^2_{H^1} \leq 3\tilde K^2$ is satisfied. We separate the proof to two parts depending on the range of $\alpha$.
\vskip6pt
\noindent {\it The case $ 1<\alpha \leq \frac32$:}
On the one hand, gathering the estimates used in section \ref{globalintimesolution} with (\ref{bad1}--\ref{partonealpha}) and (\ref{positivity}--\ref{diff}), we obtain for all $t\in[0,\overline T)$
\begin{align*}
&\frac12\,\frac d{dt} \nr \nabla w(t)\nr_2^2 \leq \left\{-\nu +\epsilon_0+\epsilon_3+2\epsilon_1+4a\epsilon_4 + \frac{\kappa_2^2\epsilon_2^2}{\epsilon_1}\left[3\tilde K^2+K^2\right] \right\}\,\nr \Delta w(t)\nr_2^2\\
&+\left\{\frac{\mu\,\nu}{2^3\epsilon_3}-\mu + \frac{3\kappa_2^2}{4\epsilon_1}\left[\frac{1}{2^4\epsilon_2^2}+\epsilon_2^2\right]\,\tilde K^2 +2\kappa_4|b|\left[\frac{|b|}{a}\right]^\frac\beta{\alpha-\beta}+\frac{\kappa_2^2\epsilon_2^2}{2\epsilon_1}\,\nr \Delta u(t)\nr_2^2\right\}\, \nr \nabla w(t)\nr_2^2\\
&+\left\{\frac{\kappa_2^4}{\epsilon_3^3}K^4 +\epsilon_0 +4a\epsilon_4 + \frac{a\kappa^4_2\kappa_3^{8\alpha}\kappa_4^4}{2\epsilon_4^3}\,|\Omega|^{8\alpha}\left[3^{4\alpha}\tilde K^{8\alpha}+\overline K_2^{8\alpha}\right]   \right\}\,\nr w(t)\nr_{H^1}^2,
\end{align*}
where,  thanks to \eqref{mass-bound-neg} and Proposition \ref{hereisthebound}, we set
\begin{align}\label{defkbar}
&\nr u(t)\nr_{H^1}\leq \left\{K^2+ \nr u_0\nr^2_{\0}+\frac{2}{\nu}\,\left(\nr f\nr^2_{\infty,2} + \eta_1\right) \right\}^\frac12:=\overline K.
\end{align}
Observe that $\overline K$ does not depend on $\mu$. On the other hand, following paragraph \ref{leconvsubsec}, we have
\begin{align}\label{tosumup}
&\frac12\,\frac d{dt} \nr w(t)\nr_2^2  +\frac{a\kappa_0}{2-\car}\,\nr(|u|+|v|)^\alpha w\nr_2^2 \leq \left(\epsilon_5-\nu+\frac{\nu}{8\epsilon_6}\right) \,\nr \nabla w(t)\nr_2^2\nonumber\\
&\hskip30pt+ \left( \epsilon_5+\mu\epsilon_6-\mu+\frac{\kappa_1^8 K^4}{\epsilon_5^3} + |b|\tilde\kappa_0\left[\frac{2|b|\tilde\kappa_0}{a\kappa_0}\right]^\frac{\beta}{\alpha-\beta}\right)\,\nr w(t)\nr_2^2.
\end{align}
Summing the last two inequalities, using the fact that $2\mu c_0h^2\leq \nu$, and optimizing in the $\epsilon'$s, we obtain for all $t\in[0,\overline T)$
\begin{align}\label{touseforlemma}
\frac d{dt}\,\nr w(t)\nr_{H^1}^2 &+ \nu\,\nr \Delta w\nr_2^2 + (2-\car)\,a\,\kappa_o\,\nr (|u|+|v|)^\alpha\,w\nr_2^2 \nonumber\\ &\leq \left[\delta_3-\mu+\max\left\lbrace \delta_1,\, \delta_2+\frac{\nu}{2^5(3\tilde K^2 +K^2)}\,\nr \Delta u\nr_2^2\right\rbrace\right]\,\nr w\nr_{H^1}^2,
\end{align}
where
\begin{align*}
&\delta_1=\frac32\nu+\frac{2^7\kappa_1^8K^4}{3^3\nu^3} + \left[\frac{(2|b|\tilde \kappa_0)^\alpha}{(a\kappa_0)^\beta}\right]^{\frac{1}{\alpha-\beta}}, \\
&\delta_2=\,\frac{3\nu\tilde K^2}{2^6(\tilde K^2+K^2)} + \frac{2^{10}3\kappa_2^4(\tilde K^2+K^2)\tilde K^2}{\nu^3} + 4\kappa_4|b|\left[\frac{|b|}{a}\right]^{\frac\beta{\alpha-\beta}},\\
&\delta_3= \frac{2^{7}\kappa_2^4K^4}{\nu^3}+ \frac5{2^4}\,\nu + \frac{2^{15}a^4\kappa_2^4\kappa_3^{8\alpha}\kappa_4^4}{\nu^3}\,|\Omega|^{8\alpha}\left(3^{4\alpha}\tilde K^{8\alpha} +\tilde K_2^{8\alpha}\right).
\end{align*}
Now, we apply Lemma \ref{Gronwall}. We write \eqref{touseforlemma} as follows
\begin{align*}
&\frac d{dt}\,\nr  w(t)\nr_{H^1}^2
+\xi(t) \,\nr  w\nr_{H^1}^2 \leq 0,
\end{align*}
with
\[\xi(s):=\mu-\delta_3-\max\left\lbrace \delta_1,\, \delta_2+\frac{\nu}{2^5(3\tilde K^2 +K^2)}\,\nr \Delta u(s)\nr_2^2\right\rbrace.\]
Thanks to \eqref{mainone}, it holds  for all $T>0$,
\begin{align}\label{thisoneforgron}
\frac\nu T\,\int_t^{t+T}\nr \Delta\,u(s)\nr^2_{\0}\,ds \leq \frac2{\nu}\,\nr f\nr^2_{\infty,2} + \left(\mathcal A_1+1\right) \,{K}^2.
\end{align}
Setting $T=1$ leads to
\[\int_t^{t+1}\xi(s)\,ds\geq \mu-\delta_3-\max\left\lbrace \delta_1,\, \delta_2+\frac{2\,\nr f\nr^2_{\infty,2} + \nu(\mathcal A_1+1) \,{K}^2}{2^5\nu(3\tilde K^2 +K^2)}\right\rbrace.\]
Therefore, if we assume
\begin{align}\label{cond132}
\mu&>2\delta_3+2\max\left\lbrace \delta_1,\, \delta_2+\frac{2\,\nr f\nr^2_{\infty,2} + \nu(\mathcal A_1+1) \,{K}^2}{2^5\nu(3\tilde K^2 +K^2)}\right\rbrace,
\end{align}
then, we have
\begin{equation}\label{topointtoit}\liminf_{t\to\infty}\,\int_t^{t+1}\, \xi(s)\,ds \geq \frac\mu2>0\quad{\rm and}\quad \int_t^{t+1}\, \xi(s)\,ds\leq \frac{3\mu}2<+\infty.\end{equation}
Thanks to Lemma \ref{Gronwall}, there exists a nonnegative constant $\eta$ such that for all  $t\in [0,\overline T)$, it holds
\begin{align}\label{convergenceinterpolant1}
\nr w\nr_{H^1}^2 \leq \nr w(t=0)\nr_{H^1}^2 \,e^{-\eta\,t}.
\end{align}
In particular, we infer that $\nr w(t)\nr_{H^1} < 2\,\tilde K$ for all $t\in [0,\overline T)$. This implies that $\nr v(t)\nr_{H^1} < 3\,\tilde K$. We reach a contradiction with the fact that $\overline T$ is the maximal time for which $\nr v(t)\nr^2_{H^1} < 3\,\tilde K^2$ and shows that $\overline T=+\infty$. Clearly, the estimate \eqref{convergenceinterpolant1} shows the exponential convergence of $v(t)$ toward $u(t)$, as $t\to +\infty$.
\vskip6pt
\noindent {\it The case $ \frac32<\alpha < 2$:} The proof is similar to the previous one except the fact that now we use the inequality \eqref{parttwoalpha} instead of \eqref{partonealpha}. However, thanks \eqref{mainonep}, we have for all $T>0$,
\begin{align*}
\frac\nu T\,\int_t^{t+T}\nr \Delta\,v(s)\nr^2_{\0}\,ds \leq \frac2{\nu}\,\nr \varphi\nr^2_{\infty,2} + (\mathcal A_2+1) \,{P}^2,
\end{align*}
The bound involves $\mu$ (through $P$) and this will lead to the same kind of problem we pointed out in the previous section about the powers of $\mu$. To get over this difficulty, we will use different estimate than \eqref{parttwoalpha} for $v$. Indeed, the only term one has to handle is the first one of the the last line of  inequality \eqref{parttwoalpha}. Using \eqref{transfer}, it is rather easy to see that
\begin{align}\label{tosavethesituationv}
&\kappa_4\,\s\,|v|^{2\alpha}|w| |-\Delta w|\,dx\leq\left( \frac{\kappa^{8(\alpha-1)}_2\,\kappa_3^{12}\,\kappa_4^4\,\epsilon_0}{\epsilon^3}\left(\nr \nabla u\nr_2^2+\nr \nabla v\nr_2^2\right)+\epsilon\right)\, \nr \Delta w\nr_2^2 \nonumber\\
&\hskip40pt \frac{\kappa^{8(\alpha-1)}_2\,\kappa_3^{12}\,\kappa_4^4}{4\,\epsilon^3}\,\left(\left[\frac1{\epsilon_0}\right]^{\frac{2\alpha-3}{4-2\alpha}}\,\nr v\nr^{\frac{2\alpha+3}{2-\alpha}}_{H^1} + \epsilon_0 \nr v\nr^2_{H^1}+2\epsilon_0\right)\, \nr \nabla w\nr^2_{2}.
\end{align}
Let $\theta:= {\kappa^{8(\alpha-1)}_2\,\kappa_3^{12}\,\kappa_4^4}$. Therefore, collecting all the estimates and using \eqref{tosavethesituationv}, we obtain
\begin{align*}
&\frac12\,\frac d{dt} \nr w(t)\nr_{H^1}^2  +\frac{a\kappa_0}{2-\car}\,\nr(|u|+|v|)^\alpha w\nr_2^2 \leq\\
&+ \left\lbrace\epsilon_5+\epsilon_0+\mu\epsilon_6-\mu+\frac{\kappa_1^8 K^4}{\epsilon_5^3}+\frac{\kappa_2^4}{\epsilon_3^3}K^4  + |b|\tilde\kappa_0\left[\frac{2|b|\tilde\kappa_0}{a\kappa_0}\right]^\frac{\beta}{\alpha-\beta} \right\rbrace\,\nr w(t)\nr_2^2.\\
&+\left\{-\nu +\epsilon_0+\epsilon_3+2\epsilon_1 +4a\epsilon_4+\left( \frac{\kappa_2^2\epsilon_2^2}{\epsilon_1}+ \frac{2a\theta\epsilon_7}{\epsilon_4^3}\right)\left[3\tilde K^2+K^2\right] \right\}\,\nr \Delta w(t)\nr_2^2\\
&+\left\{\frac{\mu\,\nu}{2^3\epsilon_3}-\mu + \frac{3\kappa_2^2}{4\epsilon_1}\left[\frac{1}{2^4\epsilon_2^2}+\epsilon_2^2\right]\,\tilde K^2 +2\kappa_4|b|\left[\frac{|b|}{a}\right]^\frac\beta{\alpha-\beta}\right\}\, \nr \nabla w(t)\nr_2^2\\
&+\frac{a\theta}{2\epsilon_4^3}\left\{  \left[\left[\frac{3^\frac12}{ \epsilon_7}\right]^{\frac{2\alpha-3}{4-2\alpha}}\tilde K^{\frac{2\alpha+3}{2-\alpha}}+3\epsilon_7\tilde K^2\right] +\left[\overline K^{\frac{2\alpha+3}{2-\alpha}}+\overline K^2\right]+\right\}\, \nr \nabla w(t)\nr_2^2\\
&+\left\{\epsilon_5-\nu+\frac{\nu}{8\epsilon_6}+\frac{\kappa_2^4}{\epsilon_3^3}K^4 +\epsilon_0 +\left[ \frac{\kappa_2^2\epsilon_2^2}{2\epsilon_1}+ \frac{a\theta(1+2\epsilon_7)}{2\epsilon_4^3}\right]\,\nr \Delta u\nr^2_{2} \right\}\, \nr \nabla w(t)\nr_2^2.
\end{align*}
Summing-up this inequality with \eqref{tosumup} and optimizing in the $\epsilon's$, we end up with
\begin{align*}
\frac d{dt}\,\nr w(t)\nr_{H^1}^2 &+ \nu\,\nr \Delta w\nr_2^2 + (2-\car)\,a\,\kappa_o\,\nr (|u|+|v|)^\alpha\,w\nr_2^2 \nonumber\\ &\leq \left[-\mu+\max\left\lbrace \delta_1,\, \delta_2+\delta_3\,\nr \Delta u\nr_2^2\right\rbrace\right]\,\nr w\nr_{H^1}^2,
\end{align*}
where
\begin{align*}
\delta_1&= 2\nu+\frac{2(\kappa_1^8+2^6\kappa_2^4)K^4}{\nu^3} + 2|b|\tilde\kappa_0\left[\frac{2|b|\tilde\kappa_0}{a\kappa_0}\right]^\frac{\beta}{\alpha-\beta},\\
\delta_2&=\frac{2\cdot 3^7 \kappa_2^4(3\tilde K^2+K^2)\tilde K^2}{\nu}+\frac{\nu\tilde K^2}{3\tilde K^2+K^2}   +\,\tilde K^2 +2\kappa_4|b|\left[\frac{|b|}{a}\right]^\frac\beta{\alpha-\beta}\\
&+ \frac{2^63^6\theta a^4a\theta}{\nu^3}  \left\lbrace\left[\frac{2^\frac92 3^\frac92a^2\theta^\frac12(3\tilde K^2+K^2)^\frac12}{\nu^2}\right]^{\frac{2\alpha-3}{4-2\alpha}}\tilde K^{\frac{2\alpha+3}{2-\alpha}}+\frac{\nu^2}{a^2\theta^\frac12(3\tilde K^2+K^2)^\frac12}\tilde K^2\right\rbrace \\
&+ \frac{2^63^6\theta a^4a\theta}{\nu^3}  \left[\overline K^{\frac{2\alpha+3}{2-\alpha}}+\overline K^2\right]   +\nu + \frac{2^7\kappa_2^4 K^4}{\nu^3},\\
\delta_3&=\frac{2^63^6(3\tilde K^2+K^2)}{\nu^3}\left[ {\kappa_2^4}+  \theta a^4\left(1+\frac{\nu^2}{a^{2}\theta^{\frac12}(3\tilde K^2+K^2)^{\frac32}}\right)\right].
\end{align*}
As done previously, we apply Lemma \ref{Gronwall} by setting
\begin{align*}
\xi(s)&:=\mu-\max\left\lbrace \delta_1,\, \delta_2+\delta_3\,\nr \Delta u\nr_2^2\right\rbrace.
\end{align*}
Now, we set $T=1$ in \eqref{thisoneforgron} and obtain
\begin{align*}
\int_t^{t+1}\,\xi(s)\,ds &\geq \mu-\max\left\lbrace \delta_1,\, \delta_2+\delta_3\,\left( \frac2{\nu}\,\nr f\nr^2_{\infty,2} + (\mathcal A_1+1) \,{K}^2\right)\right\rbrace.
\end{align*}
In particular, the assumption on the size of $\mu$,
\begin{align}
\mu&> 2\max\left\lbrace \delta_1,\, \delta_2+\delta_3\,\left( \frac2{\nu}\,\nr f\nr^2_{\infty,2} + (\mathcal A_1+1) \,{K}^2\right)\right\rbrace,  \label{cond322}
\end{align}
guarantees \eqref{topointtoit} and
\begin{align*}
&\frac d{dt}\,\nr w(t)\nr_{H^1}^2 +\xi(t) \,\nr w\nr_{H^1}^2 \leq 0.
\end{align*}
Eventually, Lemma \ref{Gronwall} allows to reach  the contradiction with the hypothesis that $\overline T <+\infty$ exactly as done for the case $1<\alpha\leq \frac32$. As a matter of fact,  we conclude to the exponential convergence of $v(t)$ toward $u(t)$, as $t\to +\infty$, with respect to the $H^1$ norm.
\section{Proof of Theorem \ref{thmu2} }
\subsection{Existence of solutions and convergence}
\vskip6pt
\noindent
In this section, we prove Theorem \ref{thmu2}. We are concerned with the following system
\begin{equation*}\mathcal S_2:\:
\left\lbrace\begin{array}{ll}
&\partial_t\,v-\nu\,\Delta\,v +(v\cdot\nabla)\,v +\nabla\,q+a\,|v|^{2\alpha}\,v+b\,|v|^{2\beta}\,v=f + \mu(\mathcal I^2_h(u)-\mathcal I^2_h(v)),\\ \\
& \nabla\cdot v=0,\;v\vert_{t=0}=v_0, \\\\
&v(x+L,y,z,t)=v(x,y+L,z,t)=v(x,y,z+L,t) = v(x,y,z,t),\\\\
&q(x+L,y,z,t)=q(x,y+L,z,t)=p(x,y,z+L,t) =q(x,y,z,t).
\end{array}
\right.
\end{equation*}
Instead of $\mathcal S_2$, let us denote by $v=w+u$ and consider the following equivalent system for $w$

\begin{equation*}\mathcal S_w:\quad
\left\lbrace\begin{array}{ll}
&\partial_t\,w-\nu\,\Delta\,w +(w\cdot\nabla)\,u +(v\cdot\nabla)\,w +\nabla\,(p_{u}-p_{v})\nonumber \\ &\\&\hskip58pt+a\,\left(|u|^{2\alpha}u-|v|^{2\alpha}v\right)+b\,\left(|u|^{2\beta}u-|v|^{2\beta}v\right)=-\mu\,\mathcal I^2_h(w),\\ \\
& \nabla\cdot w=0,\;w\vert_{t=0}=u_0-v_0,
\end{array}
\right.
\end{equation*}
with the associated periodic boundary conditions. The local (in time) existence is readily obtained using  Fadeo-Galerkin approximation (\ref{eqpart1mu}--\ref{eqpart2mu}) with  interpolant  $\mathcal I_h^2$ instead of $\mathcal I_h^1$, and observing that in the left hand side $v$ can be replaced by $u-w$ and we leave the proof for the reader. Now, we establish the necessary {\it a priori} estimates for the global existence of solution to system $\mathcal S_w$ which in turn implies the global existence of solutions of $\mathcal S_2$. Again, we perform formal calculation using $\mathcal S_w$  for simplicity.
\vskip6pt
\noindent
In the sequel, we will use the following estimate which is the analogue of \eqref{anologue0} for $\mathcal I_h^2$ instead of $\mathcal I_h^1$
\begin{align}\label{secondintesti}
- \mu\,\s\,\mathcal I^2_h(\varphi)\cdot\,(-\Delta\,\varphi)\,dx&=\mu\,\s\,\left(\varphi-\mathcal I^2_h(\varphi)\right)\cdot\,(-\Delta\,\varphi)\,dx- \mu\,\s\,\varphi\cdot\,(-\Delta\,\varphi)\,dx\nonumber\\
&\leq  \left(\frac{\mu^2\,c_0\,h^2}{4\epsilon}-\mu\right)\,\nr \nabla\,\varphi\nr^2_2 + \left(\frac{\mu^2\,c_1\,h^4}{4\epsilon}+\epsilon\right)\,\nr\Delta \varphi\nr_2^2,
\end{align}
for all $\varphi\in H^2(\Omega)$. In Theorem \ref{thmu2}, we assume that $\nr u_0\nr_{H^1} \leq \tilde K$ and $\nr v_0\nr_{H^1} \leq \tilde K$ so that $\nr w(t=0)\nr_{H^1} \leq 2\tilde K$. Again, by continuity of $\nr w(t)\nr_{H^1}$,  there exists a short time interval $[0,\overline T)$ such that for all $t\in[0,\overline T)$, it holds $\nr w(t)\nr^2_{H^1} \leq 4\tilde K^2$. In particular, it holds $\nr v(t)\nr_{H^1} \leq 3\tilde K$ for all $t\in[0,\overline T)$. In the sequel, arguing by contradiction, we will show that, actually, we have $\overline T=+\infty$. Therefore, we assume that $\overline T$ is the the maximal finite time such that $\nr v(t)\nr^2_{H^1} \leq 9\tilde K^2$ is satisfied. On the one hand, we multiply the first equation of system $\mathcal S_w$ by $w$ and integrate over $\Omega$. Thanks to \eqref{diffequationw} with $\mathcal I_h^2$ instead of $\mathcal I_h^1$, $u$ instead of $v_1$, $v$ instead of $v_2$, and \eqref{secondintesti}, we get
\begin{align}\label{here1}
\frac12\,\frac{d}{dt}\,\nr w\nr_2^2 &+\frac{a\kappa_0}{2}\,\nr (|u|+|v|)^\alpha w\nr_2^2 \leq \frac{\mu^2c_1h^4}{4\epsilon_1}\,\nr \Delta w\nr_2^2+ \left[\epsilon_0 + \frac{\mu^2c_0h^2}{4\epsilon_1}-\nu\right]\,\nr \nabla w\nr_2^2\nonumber\\
&+ \left[ \epsilon_0 +\epsilon_1+\frac{\kappa_1^8}{\epsilon_0^3}\,\nr \nabla u\nr_2^4 +\car \,\left[\frac{2^\beta(|b|\tilde\kappa_0)^\alpha}{(a\kappa_0)^\beta}\right]^{\frac1{\alpha-\beta}}-\mu \right]\,\nr w\nr_2^2.
\end{align}
On the other hand, we use (\ref{bad1}, \ref{bad2} , \ref{long}, \ref{partonealpha}) and \eqref{secondintesti} and get for all $1<\alpha\leq\frac32$
\begin{align}\label{here21}
&\frac12\,\frac{d}{dt}\,\nr \nabla w\nr_2^2 \leq \nonumber\\
&\hskip10pt\left[  \frac{\kappa_2^4K^4}{\epsilon_2^3} +\epsilon_2 +\frac{2a\epsilon_6\epsilon_8}{\epsilon_7} +\frac{a\kappa_2^4\kappa_3^{8\alpha}\kappa^4_4\,\epsilon_6}{4\epsilon_7\epsilon_8^3}\,|\Omega|^{8\alpha}\left(\overline K^{8\alpha}+(3\tilde K)^{8\alpha} \right)\right]\,\nr w\nr^2_{H^1}\nonumber \\
&+ \left[ -\mu+\frac{\mu^2c_0 h^2}{4\epsilon_5}+\frac{9\kappa_2^2\tilde K^2}{4\epsilon_3}\left(\frac{1}{16\epsilon_4^2}+\epsilon_4^2\right)+ 2\kappa_4|b|\left[\frac{|b|}{a}\right]^{\frac\beta{\alpha-\beta}} +\frac{\kappa_2^2\epsilon_4^2}{2\epsilon_3}\,\nr \Delta u\nr_2^2  \right]\,\nr \nabla w\nr_2^2\nonumber\\
&+  \left[ -\nu+\epsilon_2+2\epsilon_3+\epsilon_5+2a\kappa_4\epsilon_6+\frac{2a\epsilon_6\epsilon_8}{\epsilon_7}+\frac{\mu^2c_1h^4}{4\epsilon_5} +\frac{\kappa_2^2\epsilon_4^2}{\epsilon_3}\,\left(K^2+9\tilde K^2\right) \right]\,\nr \Delta w\nr_2^2.
\end{align}
Now, using \eqref{here1}, the fact that $15\mu\max\{c_0,\sqrt {c_1}\}\,h^2 \leq \nu$, and optimizing in the $\epsilon's$, we obtain
\begin{align}\label{touseactually1}
&\frac{d}{dt}\,\nr w(t)\nr^2_{H^1}+\nu \nr \Delta w(t)\nr_2^2 \nonumber \\&\leq \left[\delta_3-\mu+\max\left\{\delta_2,\delta_3+\frac\nu{15(K^2+9\tilde K^2)}\,\nr \Delta u\nr_2^2\right\}\right]\,\nr w(t)\nr^2_{H^1},
\end{align}
where
\begin{align*}
\delta_3&:= \frac4{15}\nu+\frac{2\cdot15^3\kappa_2^4 K^4}{\nu^3}  + \frac{\nu\kappa_2^4\kappa_3^{8\alpha}\kappa_4^4}{2^3 15},\\
\delta_2&:=2\nu+\frac{2\cdot 15^3\kappa_1^8K^4}{\nu^3} + \car \,\left[\frac{(2|b|\tilde\kappa_0)^\alpha}{(a\kappa_0)^\beta}\right]^{\frac1{\alpha-\beta}},\\
\delta_1&:= 4\kappa_4|b|\left[\frac{|b|}{a}\right]^{\frac\beta{\alpha-\beta}} + \frac{9\cdot 15^2 \kappa_2^4(K^2+9\tilde K^2)\tilde K^2}{2^5\nu^3} + \frac{\nu\tilde K^2}{10(K^2+9\tilde K^2)}.
\end{align*}
Next, we apply Lemma \ref{Gronwall} by setting
\[\xi(s):= \mu -\delta_3 -\max\left\{\delta_2,\delta_3+\frac\nu{15(K^2+9\tilde K^2)}\,\nr \Delta u\nr_2^2\right\}\]
Thanks to \eqref{mainone}, setting $T=1$ leads to
\[\int_t^{t+1}\xi(s)\,ds\geq \mu -\delta_3 -\max\left\{\delta_2,\delta_3+\frac{2\,\nr f\nr^2_{\infty,2} + \nu(\mathcal A_1+1) \,{K}^2}{15(K^2+9\tilde K^2)}\right\} \]
This, if we assume
\begin{align}\label{toputinthm1}
\mu> 2\delta_3 +2\max\left\{\delta_2,\delta_3+\frac{2\,\nr f\nr^2_{\infty,2} + \nu(\mathcal A_1+1) \,{K}^2}{15(K^2+9\tilde K^2)}\right\},
\end{align}
then \eqref{topointtoit} holds true. In particular, thanks to Lemma \ref{Gronwall}, the contradiction leading to the fact that $\overline T=+\infty$ is obtained as in section \ref{convergenceinH1norminterpolant1}. Thus, we conclude to the global existence of solution to $\mathcal S_w$ which in turn infers the global existence of solution to $\mathcal S_2$ since $v=u-w$. Moreover, the convergence at exponential rate of $v(t)$ toward $u(t)$, as $t$ goes to $+\infty$, in the $H^1$ norm follows.
\vskip6pt
\noindent
Now, we handle the case $\frac32< \alpha<2$. Using  (\ref{bad1}, \ref{bad2} , \ref{long}, \ref{partonealpha}, \ref{tosavethesituationv}) and \eqref{secondintesti}, we get
\begin{align}\label{here22}
&\frac12\,\frac{d}{dt}\,\nr \nabla w\nr_2^2 \leq \nonumber\\
&\hskip10pt\left[  \frac{\kappa_2^4K^4}{\epsilon_2^3} +\epsilon_2 +\frac{a\epsilon_6\epsilon_8}{\epsilon_7} +\frac{a\kappa_2^4\kappa_3^{8\alpha}\kappa^4_4\,\overline K^{8\alpha}\epsilon_6}{4\epsilon_7\epsilon_8^3}\,|\Omega|^{8\alpha}\right]\,\nr w\nr^2_{H^1}\nonumber \\
&+ \left[ -\mu+\frac{\mu^2c_0 h^2}{4\epsilon_5}+\frac{9\kappa_2^2\tilde K^2}{4\epsilon_3}\left(\frac{1}{16\epsilon_4^2}+\epsilon_4^2\right)+ 2\kappa_4|b|\left[\frac{|b|}{a}\right]^{\frac\beta{\alpha-\beta}} +\frac{\kappa_2^2\epsilon_4^2}{2\epsilon_3}\,\nr \Delta u\nr_2^2  \right]\,\nr \nabla w\nr_2^2\nonumber\\
&+  \frac{a\kappa^{8(\alpha-1)}_2\,\kappa_3^{12}\,\kappa_4^4\epsilon_6}{4\epsilon_7\epsilon^3_{10}}\,\left(\left[\frac1{ \epsilon_9}\right]^{\frac{2\alpha-3}{4-2\alpha}}\,(3\tilde K)^{\frac{2\alpha+3}{2-\alpha}}_{H^1} + 3\epsilon_9\tilde K^2+2\epsilon_9\nr\Delta u\nr_2^2\right)\, \nr \nabla w\nr^2_{2} \nonumber\\
& + \left[\frac{a\kappa^{8(\alpha-1)}_2\,\kappa_3^{12}\,\kappa_4^4\epsilon_6\epsilon_9}{\epsilon_7\epsilon_{10}^3}\left(K^2+9\tilde K^2\right)+\epsilon_{10}\,\right]\nr \Delta w\nr_2^2\nonumber\\
&+  \left[ -\nu+\epsilon_2+2\epsilon_3+\epsilon_5+2a\kappa_4\epsilon_6+\frac{a\epsilon_6\epsilon_8}{\epsilon_7}+\frac{\mu^2c_1h^4}{4\epsilon_5} +\frac{\kappa_2^2\epsilon_4^2}{\epsilon_3}\,\left(K^2+9\tilde K^2\right) \right]\,\nr \Delta w\nr_2^2.
\end{align}
As in the previous case, using \eqref{here1}, and the fact that $15\mu\max\{c_0,\sqrt {c_1}\}\,h^2 \leq \nu$, and optimizing in the $\epsilon's$, we obtain
\begin{align}\label{touseactually2}
&\frac{d}{dt}\,\nr w(t)\nr^2_{H^1} +\nu \nr \Delta w(t)\nr_2^2\nonumber \\ &\leq \left[\delta_3-\mu+\max\left\{\delta_2,\delta_3+\frac\nu{15}\left(1+ \frac{1}{2(K^2+3\tilde K^2)}\right)\,\nr \Delta u\nr_2^2\right\}\right]\,\nr w(t)\nr^2_{H^1},
\end{align}
where
\begin{align*}
\delta_3&:= \frac4{15}\nu+\frac{2\cdot15^3\kappa_2^4 K^4}{\nu^3}  + \frac{\nu^{13}\overline K^{8\alpha}}{4^615^{12}\kappa_2^{4(8\alpha-7)}\kappa_3^{48}\kappa_4^{12}},\\
\delta_2&:=2\nu+\frac{2\cdot 15^3\kappa_1^8K^4}{\nu^3} + \car \,\left[\frac{(2|b|\tilde\kappa_0)^\alpha}{(a\kappa_0)^\beta}\right]^{\frac1{\alpha-\beta}},\\
\delta_1&:= 4\kappa_4|b|\left[\frac{|b|}{a}\right]^{\frac\beta{\alpha-\beta}} + \frac{6 \kappa_2^4(K^2+9\tilde K^2)\tilde K^2}{\nu^3} + \frac{\nu\tilde K^2}{6\cdot 5^3(K^2+9\tilde K^2)}.
\end{align*}
The rest of the proof is the same as in the previous case and leads to the following condition on $\mu$
\begin{align}\label{toputinthm2}
\mu> 2\delta_3 +2\max\left\{\delta_2,\delta_3+\frac\nu{15}\left(1+ \frac{2{\nu}\,\nr f\nr^2_{\infty,2} + \nu(\mathcal A_1+1) \,{K}^2}{2\nu(K^2+9\tilde K^2)}\right)\right\}.
\end{align}
In the previous section, we used the fact that $\nabla u,\,\nabla v\, \in L^\infty (\mathbb R^+;L^2)$ in order to show that $\partial_t u,\partial_t v \in L_{\rm loc}^2(\mathbb R^+; L^2(\Omega))$ if the initial data is in $L^{2\alpha+2}(\Omega)$. In this section, we need to use a different estimate instead of \eqref{toseeitlater}. Indeed, thanks to \eqref{cumbersome}, one can see that the following inequality holds
\begin{align*}
&\left|\s\,\partial_t v\,v \cdot \nabla v\,dx \right| \leq \frac\epsilon{4\epsilon_0\,}\, \nr|v|^{\alpha}\,\nabla\,v\nr^{2}_{\0} + \frac{\epsilon^{\frac1{1-\alpha}}}{4\epsilon_0} \,\nr\nabla\,v\nr^{2}_{\0}  +\epsilon_0\,\nr\partial_tv\nr^2_{\0}.
\end{align*}
Clearly, the term $\epsilon_0\,\nr\partial_tv\nr^2_{\0}$ can be absorbed in the left hand side of the equivalent expression of \eqref{themainfortimederivative} for $v$ and the remaining terms are obviously integrable in time. The other new term compared to the ones of \eqref{themainfortimederivative} is
\begin{align*}
\s \mathcal I_h^2(v)\,\partial_t v \,dx &\leq \epsilon \, \nr \partial_t v\nr_2^2 +\frac1{4\epsilon} \,\nr\mathcal I_h^2(v)\nr_2^2.
\end{align*}
But, thanks to \eqref{b2}, we know that
\begin{align*}
\nr \mathcal I^2_h(v)\nr^2_{{\0}} &\leq 2\,\nr v- \mathcal I^2_h(v)\nr_{{\0}}+2\,\nr u\nr^2_{{\0}}\leq 2\,{c_0}\,h^2\,\nr \nabla u\nr^2_2+2\,{c_1}\,h^4\,\nr \Delta u\nr^2_2 + \nr u\nr^2_2.
\end{align*}
Therefore, we deduce that $\nr \mathcal I^2_h(v)\nr^2_{{\0}} $ is integrable in time.

\subsection{Continuous dependence on the initial data and uniqueness of solutions}\label{uniquesection}
\vskip6pt
\noindent
The last point to make clear is the uniqueness and the continuous dependence on the initial data for system $\mathcal S_2$. This can be shown using either $\mathcal S_w$ or $\mathcal S_2$. We use the latter system. Let $u$ be  a strong solution of system \eqref{oursystem},\eqref{pbc} and $v_1$ and $v_2$ two solutions of $\mathcal S_2$ and $w=v_1-v_2$. In order to prove the continuous dependence of the solutions on the initial data, and therefore the uniqueness, we proceed as in section \ref{convergenceinH1norminterpolant1}. Clearly, $w$ satisfies  \eqref{diffequationw} and \eqref{diffequationwzw} with $\mathcal I_h^2$ instead of $\mathcal I_h^1$, and $v_1$ and $v_2$ instead of $u$ and $v$ respectively in \eqref{diffequationwzw}. Therefore, we have clearly
\begin{align}\label{here1}
\frac12\,\frac{d}{dt}\,\nr w\nr_2^2 &+\frac{a\kappa_0}{2}\,\nr (|v_1|+|v_2|)^\alpha w\nr_2^2 \leq \frac{\mu^2c_1h^4}{4\epsilon_1}\,\nr \Delta w\nr_2^2+ \left[\epsilon_0 + \frac{\mu^2c_0h^2}{4\epsilon_1}-\nu\right]\,\nr \nabla w\nr_2^2\nonumber\\
&+ \left[ \epsilon_0 +\epsilon_1+\frac{\kappa_1^8}{\epsilon_0^3}\,\nr \nabla v_1\nr_2^4 +\car \,\left[\frac{2^\beta(|b|\tilde\kappa_0)^\alpha}{(a\kappa_0)^\beta}\right]^{\frac1{\alpha-\beta}}-\mu \right]\,\nr w\nr_2^2.
\end{align}
Now, on the one hand, for all $1<\alpha\leq \frac32$, using (\ref{bad1}, \ref{thisfirst} (with $\epsilon=1$), \ref{long}, \ref{partonealpha}) and \eqref{secondintesti}, we have
\begin{align}\label{here2}
\frac12\,\frac{d}{dt}\,\nr \nabla w\nr_2^2 &\leq \nonumber\\
&\hskip10pt\left[  \frac{\kappa_2^4}{\epsilon_2^3}\,\nr \nabla v_1\nr_2^4 +\epsilon_2 +\frac{2a\epsilon_6}{\epsilon_5} +\frac{a\kappa_2^4\kappa_3^{8\alpha}\kappa^4_4}{4\epsilon_5\epsilon_6^3}\,|\Omega|^{8\alpha}\left(\nr v_1\nr^{8\alpha}_{H^1} +\nr v_2\nr^{8\alpha}_{H^1} \right)\right]\,\nr w\nr^2_{H^1}\nonumber \\
&+ \left[ -\mu+\frac{\mu^2c_0 h^2}{4\epsilon_4}+\frac{\kappa_2^2}{2\epsilon_3}\,\nr v_2\nr_{H^1}^2+ 2\kappa_4|b|\left[\frac{|b|}{a}\right]^{\frac\beta{\alpha-\beta}} +\frac{\kappa_2^2}{4\epsilon_3}\,\nr \Delta v_2\nr_2^2  \right]\,\nr \nabla w\nr_2^2\nonumber\\
&+  \left[ -\nu+\epsilon_2+2a\kappa_4\epsilon_5+\epsilon_4+2\epsilon_3+\frac{2a\epsilon_6}{\epsilon_5}+\frac{\mu^2c_1h^4}{4\epsilon_4}  \right]\,\nr \Delta w\nr_2^2.
\end{align}
On the other hand,  for all $\frac32< \alpha<2$, using (\ref{bad1}, \ref{thisfirst} (with $\epsilon=1$), \ref{long}, \ref{parttwoalpha} (with $\tilde\epsilon=1$)) and \eqref{secondintesti}, we obtain
\begin{align}\label{here3}
\frac12\,\frac{d}{dt}\,\nr \nabla w\nr_2^2 &\leq \left[  \frac{\kappa_2^4}{\epsilon_2^3}\,\nr \nabla v_1\nr_2^4 +\epsilon_2  \right]\,\nr w\nr^2_{H^1}\nonumber \\
&+ \left[ -\mu+\frac{\mu^2c_0 h^2}{4\epsilon_4}+\frac{\kappa_2^2}{2\epsilon_3}\,\nr v_2\nr_{H^1}^2+ 2\kappa_4|b|\left[\frac{|b|}{a}\right]^{\frac\beta{\alpha-\beta}} +\frac{\kappa_2^2}{4\epsilon_3}\,\nr \Delta v_2\nr_2^2  \right]\,\nr \nabla w\nr_2^2\nonumber\\
&+ \frac{a\kappa_2^{8(\alpha-1)}\kappa_3^{12}\kappa_4^4}{4\epsilon_5\epsilon_6^3}\left[ \nr v_1\nr^{\frac{2\alpha+3}{2-\alpha}}_{H^1}+\nr v_1\nr^{2}_{H^1}+\nr v_2\nr^{\frac{2\alpha+3}{2-\alpha}}_{H^1}+\nr v_2\nr^{2}_{H^1} \right]\,\nr \nabla w\nr_2^2\nonumber\\
&+ \frac{a\kappa_2^{8(\alpha-1)}\kappa_3^{12}\kappa_4^4}{4\epsilon_5\epsilon_6^3}\left[ \nr \Delta v_1\nr^{2}_{2}+\nr \Delta v_2\nr^{2}_{2} \right]\,\nr \nabla w\nr_2^2\nonumber\\
&+  \left[ -\nu+\epsilon_2+2a\kappa_4\epsilon_5+\epsilon_4+2\epsilon_3+\frac{2a\epsilon_6}{\epsilon_5}+\frac{\mu^2c_1h^4}{4\epsilon_4}  \right]\,\nr \Delta w\nr_2^2.
\end{align}
Therefore, adding \eqref{here1} to \eqref{here2}  and \eqref{here1} to \eqref{here3}, and optimizing in the $\epsilon'$s, we obtain using the fact that $15\mu\max\{c_0,\sqrt {c_1}\}\,h^2 \leq \nu$
\begin{align*}
\frac{d}{dt}\,\nr w(t)\nr^2_{H^1}+ \nu\,\nr \Delta w\nr_2^2&+\frac{a\kappa_0}{2}\,\nr (|v_1|+|v_2|)^\alpha w\nr_2^2 \\&\leq \left[\delta_3(t)-\mu+\max\{\delta_2(t),\delta_1(t)\}\right]\,\nr w(t)\nr^2_{H^1},
\end{align*}
where for all $1<\alpha\leq\frac32$, we set
\begin{align*}
\delta_3(t)&:= \frac{2\cdot 15^3\kappa_2^4}{\nu^3} \,\nr \nabla v_1(t)\nr_2^4+ \frac{2(1+a\kappa_4)\,\nu}{15} \\ &+\frac{4\cdot 15^7a^3\kappa_2^4\kappa_3^{8\alpha}\kappa^4_5}{\nu^7}\,|\Omega|^{8\alpha}\left(\nr v_1(t)\nr^{8\alpha}_{H^1} +\nr v_2(t)\nr^{8\alpha}_{H^1} \right) ,\\
\delta_2(t)&:=2\nu+\frac{2\nu}{15}+ \frac{2\kappa_1^8}{\nu^3}\,\nr \nabla v_1(t)\nr_2^4 + \car \,\left[\frac{(2|b|\tilde\kappa_0)^\alpha}{(a\kappa_0)^\beta}\right]^{\frac1{\alpha-\beta}},\\
\delta_1(t)&:=\frac{15\kappa_2^2}{\nu}\,\nr v_2(t)\nr_{H^1}^2+\frac{15\kappa_2^2}{2\nu}\,\nr \Delta v_2(t)\nr_2^2+ 4\kappa_4|b|\left[\frac{|b|}{a}\right]^{\frac\beta{\alpha-\beta}} ,
\end{align*}
and for all $\frac32< \alpha<2$,
\begin{align*}
\delta_3(t)&:= \frac{2\cdot 15^3\kappa_2^4}{\nu^3} \,\nr \nabla v_1(t)\nr_2^4 + \frac{2\nu}{15},\\
\delta_2(t)&:=2\nu+\frac{2\nu}{15}+ \frac{2\kappa_1^8}{\nu^3}\,\nr \nabla v_1(t)\nr_2^4 + \car \,\left[\frac{(2|b|\tilde\kappa_0)^\alpha}{(a\kappa_0)^\beta}\right]^{\frac1{\alpha-\beta}},\\
\delta_1(t)&:= \frac{4\cdot 15^7 a^3\kappa_2^{8(\alpha-1)}\kappa_3^{12}\kappa_4^5}{\nu^7}\,\left[ \nr v_1(t)\nr^{\frac{2\alpha+3}{2-\alpha}}_{H^1}+\nr v_1(t)\nr^{2}_{H^1}+\nr v_2(t)\nr^{\frac{2\alpha+3}{2-\alpha}}_{H^1}+\nr v_2(t)\nr^{2}_{H^1} \right]\\
&+\frac{4\cdot 15^7 a^3\kappa_2^{8(\alpha-1)}\kappa_3^{12}\kappa_4^5}{\nu^7}\,\left[ \nr \Delta v_1(t)\nr^{2}_{2}+\nr \Delta v_2(t)\nr^{2}_{2} \right]\\
&+\frac{15\kappa_2^2}{\nu}\,\nr v_2(t)\nr_{H^1}^2+\frac{15\kappa_2^2}{2\nu}\,\nr \Delta v_2(t)\nr_2^2+ 4\kappa_4|b|\left[\frac{|b|}{a}\right]^{\frac\beta{\alpha-\beta}}.
\end{align*}
By virtue of Gronwall's inequality, we can write
\begin{align}\label{lastofpaper}\nr w(t)\nr_{H^1}^2 \leq \nr w(t=0)\nr_{H^1}^2\,e^{\int_0^t  \left[\delta_3(s)-\mu+\max\{\delta_2(s),\delta_3(s)\}\right] ds }.\end{align}
Next, recall that $\nr v_1(t)\nr_{H^1}$ and $\nr v_2(t)\nr_{H^1}$ are uniformly bounded. Furthermore, integrating  \eqref{touseactually1} or \eqref{touseactually2} with respect to time, it is rather easy to see that
\[\int_0^T\,\nr \Delta w(s)\nr_2^2\,ds<+\infty.\]
In particular, we infer that
\[\int_0^T\,\nr \Delta v_1(s)\nr_2^2\,ds<+\infty \quad \text{and}\quad \int_0^T\,\nr \Delta v_2(s)\nr_2^2\,ds<+\infty.\]
The continuous dependence on the initial data and the uniqueness follow from \eqref{lastofpaper}.

\vskip6pt
\section*{Acknowledgements}
\noindent E.S.T.~is thankful to the kind hospitality of KAUST where this work was started. E.S.T.~also acknowledges the partial support of the National Science Foundation through grants number DMS--1109640 and DMS--1109645. The research of P. A. Markowich and  S. Trabelsi reported in this publication was supported by the King Abdullah University of Science and Technology.

\end{document}